\documentclass[11pt]{amsart}
\usepackage{tikz,amsthm,amsmath,amstext,amssymb,amscd,epsfig,euscript, mathrsfs, dsfont,pspicture,multicol,graphpap,graphics,graphicx,times,enumerate,subfig,sidecap,wrapfig,color}
\usepackage{amsmath}
\usepackage{amssymb}
\usepackage{pgf}

\usepackage[colorlinks,citecolor=red]{hyperref}

\newcommand{\R}{{\mathbb R}}       
\newcommand{\N}{{\mathbb N}}       %

\newcommand{\HH}{{\mathcal H}}

\newcommand{\RR}{{\mathcal R}}

\newcommand{\EE}{{\mathcal E}}

\newcommand{\diam}{\mathop{\rm diam}}
\newcommand{\dist}{{\rm dist}}

\newcommand{\rf}[1]{{(\ref{#1})}}

\newcommand{\supp}{\operatorname{supp}}

\newcommand{\vphi}{{\varphi}}
\newcommand{\ve}{{\varepsilon}}
\newcommand{\vv}{{\vspace{2mm}}}
\newcommand{\vvv}{\vspace{4mm}}
\newcommand{\wt}[1]{{\widetilde{#1}}}

\newcommand{\pom}{{\partial \Omega}}

\def\Xint#1{\mathchoice
{\XXint\displaystyle\textstyle{#1}}%
{\XXint\textstyle\scriptstyle{#1}}%
{\XXint\scriptstyle\scriptscriptstyle{#1}}%
{\XXint\scriptscriptstyle\scriptscriptstyle{#1}}%
\!\int}
\def\XXint#1#2#3{{\setbox0=\hbox{$#1{#2#3}{\int}$ }
\vcenter{\hbox{$#2#3$ }}\kern-.58\wd0}}

\def\avint{\Xint-}

\newtheorem{theorem}{Theorem}[section]
\newtheorem{lemma}[theorem]{Lemma}
\newtheorem{remark}[theorem]{Remark}
\newtheorem{corollary}[theorem]{Corollary}

\newtheorem*{lemma*}{Lemma}
\newtheorem*{theorem*}{Theorem}
\newtheorem*{theorema}{Theorem A}
\newtheorem*{theoremb}{Theorem B}

\theoremstyle{definition}
\newtheorem{definition}[theorem]{Definition}

\theoremstyle{remark}
\newtheorem{rem}[theorem]{\bf Remark}

\numberwithin{equation}{section}

\newcommand{\cnj}[1]{\overline{#1}}
\newcommand{\RRem}{\begin{rem}}
\newcommand{\erem}{\end{rem}}

\def\d{\partial}

\makeatletter
\def\@tocline#1#2#3#4#5#6#7{\relax
  \ifnum #1>\c@tocdepth 
  \else
    \par \addpenalty\@secpenalty\addvspace{#2}%
    \begingroup \hyphenpenalty\@M
    \@ifempty{#4}{%
      \@tempdima\csname r@tocindent\number#1\endcsname\relax
    }{%
      \@tempdima#4\relax
    }%
    \parindent\z@ \leftskip#3\relax \advance\leftskip\@tempdima\relax
    \rightskip\@pnumwidth plus4em \parfillskip-\@pnumwidth
    #5\leavevmode\hskip-\@tempdima
      \ifcase #1
       \or\or \hskip 1em \or \hskip 2em \else \hskip 3em \fi%
      #6\nobreak\relax
    \dotfill\hbox to\@pnumwidth{\@tocpagenum{#7}}\par
    \nobreak
    \endgroup
  \fi}
\makeatother

\def\Cap{\mathop\mathrm{Cap}}

\def\cF{{\mathscr{F}}}

\def\cH{{\mathcal{H}}}

\def\cM{{\mathscr{M}}}

\def\bR{{\mathbb{R}}}

\def\bZ{{\mathbb{Z}}}

\def\bN{{\mathbb{N}}}

\newcommand{\ps}[1]{\left( #1 \right)}
\newcommand{\bk}[1]{\left[ #1 \right]}
\newcommand{\ck}[1]{\left\{#1 \right\}}
\newcommand{\av}[1]{\left| #1 \right|}

\newcommand{\divv}{{\text{{\rm div}}}}

\def\loc{\textrm{loc}}
\def\grad{\nabla}
\def\loc{\textrm{loc}}
\def\Tan{\textrm{Tan}}
\def\bR{\mathbb{R}}
\def\lec{\lesssim}

\def\ext{\textrm{ext}}

\newcommand\eqn[1]{\eqref{e:#1}}

\newcommand\Theorem[1]{Theorem \ref{t:#1}}
\newcommand\Lemma[1]{Lemma \ref{l:#1}}

\begin{document}

\title[Mutual absolute continuity of interior and exterior harmonic]{Mutual absolute continuity of interior and exterior harmonic measure implies
rectifiability}

\author[Azzam]{Jonas Azzam}
\address{Jonas Azzam
\\
Departament de Matem\`atiques
\\
Universitat Aut\`onoma de Barcelona
\\
Edifici C Facultat de Ci\`encies
\\
08193 Bellaterra (Barcelona), Catalonia
}
\email{jazzam@mat.uab.cat}

\newcommand{\jonas}[1]{\marginpar{\color{magenta} \scriptsize \textbf{Jonas:} #1}}

\author[Mourgoglou]{Mihalis Mourgoglou}

\address{Mihalis Mourgoglou
\\
Departament de Matem\`atiques\\
 Universitat Aut\`onoma de Barcelona
\\
Edifici C Facultat de Ci\`encies
\\
08193 Bellaterra (Barcelona), Catalonia
}
\email{mourgoglou@mat.uab.cat}

\author[Tolsa]{Xavier Tolsa}
\address{Xavier Tolsa
\\
ICREA and Departament de Matem\`atiques
\\
Universitat Aut\`onoma de Barcelona
\\
Edifici C Facultat de Ci\`encies
\\
08193 Bellaterra (Barcelona), Catalonia
}
\email{xtolsa@mat.uab.cat}

\subjclass[2010]{31A15,28A75,28A78}
\thanks{The three authors were supported by the ERC grant 320501 of the European Research Council (FP7/2007-2013). X.T. was also supported by 2014-SGR-75 (Catalonia), MTM2013-44304-P (Spain), and by the Marie Curie ITN MAnET (FP7-607647).}

\begin{abstract}
We show that, for disjoint domains in the Euclidean space whose boundaries satisfy a non-degeneracy condition, mutual absolute continuity of their harmonic measures implies  absolute continuity with respect to surface measure and rectifiability in the intersection of their boundaries.
\end{abstract}

\maketitle

\tableofcontents

\section{Introduction}
The relationship between the properties of harmonic measure and the geometry of its support has attracted the attention of many 
mathematicians. In this paper we study a two-phase problem in connection with this topic.
More precisely, we show that, for disjoint domains in $\R^{n+1}$ whose boundaries satisfy the so called $\Delta$-regularity condition, mutual absolute continuity of their harmonic measures implies  absolute continuity with respect to surface measure and $n$-rectifiability in the intersection of their boundaries.
This result solves a conjecture of Chris Bishop from 1990, under the $\Delta$-regularity 
assumption. See Conjecture 8 from
\cite{Bishop-questions} or Section 6 from \cite{Bishop-arkiv}.
 
To state our results in detail we need to introduce some notation. Given a domain (i.e., an open
and connected set) $\Omega\subset\R^{n+1}$, with $n\geq 2$, we denote by $\omega^x$ or $\omega_\Omega^x$ its harmonic measure with respect to a pole $x\in\Omega$. If the precise pole of the harmonic measure
is not relevant, we may also write just $\omega$ or $\omega_\Omega$.

For the precise notion of $\Delta$-regularity, we refer the reader to Definition 
\ref{defdelta} below. However, we mention that by a result of Ancona \cite{Anc86}
this is equivalent to the more known ``capacity density condition'' (CDC), and by a deep theorem of Lewis 
\cite{Lewis}, it follows that $\Omega$ is $\Delta$-regular
if and only if there exists some $\ve>0$ and some $R>0$ such that
$$\HH_\infty^{n-1+\ve}(B(x,r)\setminus \Omega) \approx r^{n-1+\ve}
\quad \mbox{ for all $x\in\partial\Omega$ and all $0<r\leq R$,}$$
where $\HH_\infty^s$ stands for the $s$-dimensional Hausdorff content.
We also remark that, in particular, the nontangentially accessible domains of Jerison and Kenig 
\cite{JK82} are examples of $\Delta$-regular domains. More generally, it is easy to check that a domain $\Omega$ is also $\Delta$-regular if it just satisfies a two sided corkscrew condition, that is, any ball $B(x,r)$ centered at $x\in\d\Omega$, $0<r\leq R$, contains
two balls $B_1\subset B(x,r)\cap\Omega$ and $B_2\subset B(x,r)\setminus\Omega$ with $r(B_1)=r(B_2)\approx r$.

Recall that a point $x\in \bR^{n}$ is an {\it $n$-dimensional tangent} for a set $E\subset \bR^{n+1}$ if there is an $n$-dimensional plane $V$ containing $x$ so that 
\[
\lim_{r\rightarrow 0}\sup_{\zeta\in B(x,r)\cap E} \frac{\dist(\zeta,V)}{r}=0.\]

Our main result is the following:

\begin{theorem}\label{t:main}
For $n\geq 2$, let $\Omega^+\subset \R^{n+1}$ be  open and let $\Omega^- = \bigl(\,\overline{\Omega^+}\,\bigr)^c$. Assume that $\Omega^+,\Omega^-$ are both connected and $\Delta$-regular and $\partial\Omega^+ = \partial\Omega^-$.
Let $\omega^\pm$ be the respective harmonic measures of $\Omega^\pm$.
Let $E\subset \d\Omega^+$ be a Borel set and let $T$ the set of tangent points for $\d\Omega^+$. Then $\omega^+\perp \omega^-$ on $E$ if and only if $\cH^{n}(E\cap T)=0$. Further, if $\omega^{+}\ll\omega^{-}\ll\omega^{+}$ on $E$, then $E$ contains an $n$-rectifiable subset $F$ upon which $\omega^\pm$ are mutually absolutely continuous with respect to $\cH^{n}$. 
\end{theorem}

From this result we derive other local versions for two sided $\Delta$-regular domains.
 We say that a domain $\Omega\subset\R^{n+1}$ is two sided $\Delta$-regular if both $\Omega $ and
 ${\rm ext}(\Omega) := \bigl(\,\overline{\Omega}\,\bigr)^c$ are connected and
$\Delta$-regular.

\begin{corollary}\label{coro1}
For $n\geq 2$, let $\Omega^1\subset \R^{n+1}$ be an open domain. Suppose that $\Omega^1$ is two sided $\Delta$-regular and that
$\partial\Omega^1 = \partial({\rm ext}(\Omega^1))$.
Let $\Omega^2\subset \R^{n+1}$ be a domain disjoint from $\Omega^1$.
For $i=1,2$, let $\omega^i$ be the respective harmonic measures of $\Omega^i$.
Let $E\subset \d\Omega^{1}\cap \d\Omega^{2}$ be a Borel set. If $\omega^1\ll\omega^2\ll\omega^1$ on $E$, then $E$ contains an $n$-rectifiable subset $F$ upon which $\omega^1$ and $\omega^2$ are mutually absolutely continuous with respect to $\cH^{n}$. 
\end{corollary}

\begin{corollary}\label{coro2}
For $n\geq 2$, let $\Omega^1\subset \R^{n+1}$ be an open domain. Suppose that $\Omega^1$ is two sided $\Delta$-regular and that
$\partial\Omega^1 = \partial({\rm ext}(\Omega^1))$.
Let $\Omega^2\subset \R^{n+1}$ be a domain disjoint with $\Omega^1$.
For $i=1,2$, let $\omega^i$ be the respective harmonic measures of $\Omega^i$.
Let $E\subset \d\Omega^{1}\cap \d\Omega^{2}$ be relatively open both in $\partial\Omega^1$ and
$\partial\Omega^2$. Let $T$ be the set of tangent points for $\d\Omega^1$. Then $\omega^1\perp \omega^2$ on $E$ if and only if $\cH^{n}(E\cap T)=0$.
\end{corollary}

The referee shared with us an example that shows Corollary \ref{coro2} does not hold for general domains. Construct a domain in $\R^{3}$ as follows. Let $\ve_{k}\downarrow 0$ and 
\[
\Omega=\R^{3}_{+}\backslash \bigcup_{k=0}^{\infty} \bigcup_{\xi\in 2^{-k}(\bZ\times \bZ\times \{0\})} \cnj{(B(\xi,\ve_{k} 2^{-k})\cap \bR^{2})\times 2^{-k}}.\]
If $\ve_{k}\downarrow 0$ fast enough, then $\omega_{\Omega}$ and $\omega_{\cnj{\Omega}^{c}}$ is mutually absolutely continuous on $\d\Omega\cap \R^{2}$ but no point in $\d\Omega$ is a tangent point. Thus, a condition like $\Delta$-regularity is necessary for the result to hold. 

In the case of the plane ($n=1$), the above results are already known, and basically they follow from the
following nice theorem of Chris Bishop \cite{Bishop-arkiv}:

\begin{theorema}
Let $\Omega^1$ and $\Omega^2$ be disjoint domains in $\R^2$ and let $\omega^1$ and $\omega^2$ 
be their harmonic measures. Then $\omega^1\perp \omega^2$ if and only if the set of points in 
$\partial\Omega^1\cap \d\Omega^2$ satisfying
a weak double cone condition with respect to $\Omega^1$ and $\Omega^2$ has zero 1-dimensional
Hausdorff measure. Moreover, if $\omega^1$ and $\omega^2$ are mutually absolutely continuous on a 
Borel set $E\subset \partial\Omega^1\cap \d\Omega^2$, then $E$ contains a $1$-rectifiable subset $F$ upon which $\omega^1$ and $\omega^2$ are mutually absolutely continuous with respect to $\cH^{1}$. 
\end{theorema}

For the definition of the weak double cone condition we refer the reader to the original 
paper of Bishop \cite{Bishop-arkiv}.
We also mention that in the particular case when $\partial\Omega^1\cap \d\Omega^2$ is a
Jordan arc in the plane, the preceding result is a direct consequence of a previous work
by Bishop, Carleson, Garnett and Jones \cite{BCGJ88}.

The main obstacle for the challenge of extending the aforementioned results of Bishop and Bishop, Carleson, Garnett and Jones 
to higher dimensions arises from the fact that the arguments in these works rely heavily on the use of complex analysis. Up to now, the main contribution on this objective was the work of Kenig, Preiss, and Toro \cite{KPT09}, whose result we paraphrase below.

\begin{theoremb} 
Let $\Omega^+$ and $\Omega^-={\rm ext}(\Omega^+)$ be two NTA domains in $\bR^{n+1}$, $n\geq 2$, and $\omega^{\pm}=\omega_{\Omega_{\pm}}^{x_{\pm}}$ their harmonic measures. Then $\d\Omega^+=\Gamma_{g}\cup \Gamma_{b}\cup N\cup S$, where
\begin{enumerate}
\item $\omega^{+}|_{S}\perp \omega^{-}|_{S}$,
\item $\omega^{\pm}(N)=0$,
\item $\dim_\HH (\Gamma_{b}\cup \Gamma_{g})=n$,
\item each $\xi\in \Gamma_{b}\cup \Gamma_{g}$ is an $n$-dimensional tangent point for $\d\Omega$,
\item $\omega^{+}|_{\Gamma_{g}} \ll \cH^{n}|_{\Gamma_{g}}\ll \omega^{-}|_{\Gamma_{g}}\ll \omega^{+}|_{\Gamma_{g}}$, and
\item if $E\subset \Gamma_{b}$ is Borel with $\omega^{\pm}(E)>0$, then $\cH^{n}(E)=\infty$.
\end{enumerate}
\end{theoremb}

Again, see \cite{JK82} for the definition of NTA domains. In (3), $\dim_\HH$ stands for the Hausdorff dimension. Let us remark that in \cite{KPT09} it was also proved that $\Gamma_{b}\cup \Gamma_{g}$
is $n$-rectifiable and $\Gamma_{b}=\varnothing$ under the assumption that $\d\Omega$ has locally finite $\HH^n$-measure (in fact, the whole boundary is $n$-rectifiable in this case by the Besicovitch-Federer projection theorem). An immediate consequence of Theorem \ref{t:main} is that in the preceding result of Kenig,
Preis and Toro we can assert that $\Gamma_b =\varnothing$, because the set where $\omega^+$ and
$\omega^-$ are mutually absolutely continuous satisfies the same property as $\Gamma_g$ in (5), up to
a set of null harmonic measure $\omega^\pm$.

The proof of Theorem B above is a beautiful marriage of techniques from partial differential equations and geometric measure theory. The crucial tools are the theory of nontangentially accessible domains introduced by Jerison and Kenig \cite{JK82}, the monotonicity formula of Alt, Caffarelli, and Friedman \cite{ACF84}, the theory of tangent measures introduced by Preiss \cite{Pr87}, and the blow up techniques for harmonic measures at infinity for unbounded NTA domains due to Kenig and Toro \cite{KT99,KT06}.

The authors used Theorem B to resolve a conjecture put forth by Lewis, Verchota, and Vogel. In \cite{Wol95}, Wolff showed that there are two-sided NTA domains in $\bR^{3}$ whose harmonic measures may have dimensions strictly bigger or smaller than $2$. In \cite{LVV05},  Lewis, Verchota, and Vogel generalized this to higher dimensions and showed that there are two-sided NTA domains in $\bR^{n+1}$ for any $n\geq 2$ whose interior and exterior harmonic measures can have dimensions either below or above $n$ (in any combination). They also conjectured that there should be such a two-sided NTA domain whose harmonic measures ---in addition to having fractional dimensions--- should be mutually absolutely continuous. However, a consequence of Theorem B is that the dimension of the harmonic measures is equal to $n$ if mutual absolute continuity occurs. Additionally, by Corollary \ref{coro2}, in this case the harmonic measures are concentrated on a countable union of Lipschitz graphs, and hence on a set of $\sigma$-finite $\cH^{n}$-measure.


Our arguments for the proof of Theorem \ref{t:main} improve on the techniques of the aforementioned 
work of Kenig, Preiss and Toro and include a new set of ideas involving the $n$-dimensional Riesz transform.
The connection between the Riesz transform and harmonic measure is due to the fact that the Riesz kernel is the gradient of the Newtonian potential, and the relationship between the Riesz transform and rectifiability is a subject that has been in constant development for the last twenty years
and has culminated in the solution of the David-Semmes conjecture by  Nazarov, Tolsa and Volberg
\cite{NToV1,NToV-pubmat}.
 For recent examples of Riesz transform techniques used to study harmonic measure, see for instance \cite{BH15,AHMMMTV15,MT15}. The arguments in the current paper use  a new recent result by
 Girela-Sarri\'on and Tolsa \cite{GT} on the connection between Riesz transforms and quantititative 
 rectifiability for general Radon measures (see Theorem \ref{teo0} below for more details).
 
One can view the works described above as sort of an endpoint case of a larger class of two phase problems where one is interested in studying the smoothness of $\partial\Omega$ in terms of the smoothness of $\frac{d\omega^+}{d\omega^-}$; in other words, better behavior of $\frac{d\omega^+}{d\omega^-}$ implies better regularity of $\d\Omega$. For example, most recently, Engelstein \cite{Engelstein} showed that for two-sided NTA domains in $\bR^{n+1}$ (Reifenberg flat if $n\geq 2$), if $\alpha\in (0,1)$, $k\geq 0$ is an integer, and $\log \frac{d\omega^+}{d\omega^-}\in C^{k,\alpha}$, then locally $\d\Omega$ is the graph of a $C^{k+1,\alpha}$ function. See also \cite{ACF84}, \cite{Badger13}, \cite{C87}, \cite{DFS}, as well as the references therein, for example.\\

The authors are very grateful to the anonymous referee for useful suggestions that improved the paper and for providing an example to show the tightness of the $\Delta$-regularity assumption.

\vv

\section{Preliminaries}

We will write $a\lesssim b$ if there is $C>0$ so that $a\leq Cb$ and $a\lesssim_{t} b$ if the constant $C$ depends on the parameter $t$. We write $a\approx b$ to mean $a\lesssim b\lesssim a$ and define $a\approx_{t}b$ similarly.

For sets $A,B\subset \bR^{n+1}$, let 
\[\dist(A,B)=\inf\{|x-y|:x\in A,y\in B\}, \;\; \dist(x,A)=\dist(\{x\},A),\]
and 
\[\diam A=\sup\{|x-y|:x,y\in A\}.\]
For a subset $A\subset\R^{n+1}$ and $0<\delta\leq \infty$ one sets
\[\cH^{n}_{\delta}(A)=\inf\ck{\sum \diam(A_i)^{n}: A\subset \bigcup A_i,\,\diam(A_i)\leq \delta}.\]
 The {\it $n$-dimensional Hausdorff measure} of $A$ is defined as
\[\cH^{n}(A)=\lim_{\delta\downarrow 0}\cH^{n}_{\delta}(A),\]
and $\cH^{n}_{\infty}(A)$ is called the {\it $n$-dimensional Hausdorff content} of $A$. See \cite[Chapter 4]{Mattila} for more details.

We recall now the notion of $n$-rectifiability and its quantitative analogue (uniform $n$-rectifiability).

\begin{definition}\label{def:rec-lip}
A Borel set $E\subset \R^{n+1}$ is {\it $n$-rectifiable} if there exist $E_{i}\subset \R^{n}$ and $f_{i}:E_{i}\rightarrow \R^{n+1}$ Lipschitz so that $\HH^n(E\backslash \bigcup_{i=1}^{\infty} f_{i}(E_{i}) )=0$. 
\end{definition}

\begin{definition}\label{def:ADR-UR}
A set $E\subset\R^{n+1}$ is {\it $n$-Ahlfors-David regular (or $n$-AD-regular)} if
\begin{equation}\label{eq:adr1}
C^{-1}r^n\leq\HH^n(E\cap B(x,r))\leq C\,r^n\quad \mbox{ for all $x\in E$ and $0<r\leq\diam(E)$}.
\end{equation}
A set $E\subset\R^{n+1}$ is {\it uniformly  $n$-rectifiable} if it is 
$n$-AD-regular and
there exist $\theta, M >0$ such that for all $x \in E$ and all $r>0$ 
there is a Lipschitz mapping $g: B_n(0,r) \subset \R^{n} \to \R^{n+1}$ with $\text{Lip}(g) \leq M$ such that$$
\HH^n (B(x,r)\cap g(B_{n}(0,r))\cap E)\geq \theta r^{n}.$$

In the case $n=1$, it is known that $E$ is uniformly $1$-rectifiable if and only if $E$ is contained in a $1$-AD-regular curve 
in $\R^{n+1}$. We will call the constants $M$, $\theta$ and $C$ in \rf{eq:adr1} the UR constants of $E$.
\end{definition}

\begin{definition}\label{def:BV}
A function $f \in L_{loc}^1(U)$ has {\it locally bounded variation} in an open set $U \subset \R^{n+1}$ and we write $f \in BV_{\loc}(U)$, if for each open set $V\Subset U$,
$$\sup \left\{ \int_V f \,\,\divv \phi \,\,d\mathcal L^{n+1}: \phi \in C^\infty_c(V;\R^{n+1}),\,\, |\phi| \leq 1 \right\} <\infty,$$
where $\mathcal L^{n+1}$ stands for the $(n+1)$-dimensional Lebesgue measure. An $\mathcal L^{n+1}$-measurable set $E \subset \R^{n+1}$ has {\it locally finite perimeter} in $U$ if $\chi_E \in BV_{loc}(U)$. Recall that the Radon measures in $\R^{n+1}$ are just the Borel measures which are locally finite (and they turn out to be inner regular).

\end{definition}
 We now state the Structure Theorem for $ BV_{loc}$ functions, whose proof can be found in \cite[p. 167]{EvansGariepy}.
\begin{theorem}\label{thm:BV-structure}
Let $f \in BV_{loc}(U)$. Then there exists a Radon measure $\mu$ on $U$ and a $\mu$-measurable function $\sigma : U  \to \R^{n+1}$ so that 
\begin{enumerate}
\item $|\sigma(x)|=1$, for $\mu$-a.e. $x \in U$ and
\item $\int_U f \,\,\divv \phi \,\,d\mathcal L^{n+1} = - \int \phi \cdot \sigma \,d\mu$, for all $\phi \in C^\infty_c(U, \R^{n+1})$.
\end{enumerate}
If $f=\chi_E$ and $E$ has locally finite perimeter in $U$, then we denote $\|\d E\|=\mu$ and $\nu_E=-\sigma $.
\end{theorem}

\begin{definition}
Let $E$ be a set of locally finite perimeter in $\R^{n+1}$ and $x\in \R^{n+1}$. The {\it reduced boundary} of $E$, which we denote by $\d^{*}E$, is the set of points $x\in \d E$ such that
\begin{enumerate}
\item  $\|\d E\|(B(x,r))>0$, for all $r>0$,
\item $\lim_{r \to 0} \frac{1}{\|\d E\|(B(x,r))}\int_{B(x,r)} \nu_E(y) \,d\|\d E\| = \nu_E(x)$, and
\item $|\nu_E(x)|=1$.
\end{enumerate}
\end{definition}

\begin{definition}
For each $x \in \d^* E$ we define the {\it hyperplane}
$$H(x)=\left\{ y \in \R^{n+1}: \nu_E(x) \cdot (y-x) =0  \right\}$$
and the {\it half-spaces}
\begin{align*}
H^+(x)&=\left\{ y \in \R^{n+1}: \nu_E(x) \cdot (y-x) \geq 0  \right\},\\
H^-(x)&=\left\{ y \in \R^{n+1}: \nu_E(x) \cdot (y-x) \leq 0  \right\}.
\end{align*}
A unit  vector $\nu_E(x)$ is called the {\it measure theoretic unit outer normal} to $E$ at $x$ if
$$
\lim_{r \to 0} \frac{\mathcal L^{n+1} ( B(x,r) \cap E \cap H^+(x))}{r^{n+1}}=0
$$
and
$$
\lim_{r \to 0} \frac{\mathcal L^{n+1} ( (B(x,r) \setminus E) \cap H^-(x))}{r^{n+1}}=0.
$$
\end{definition}

\begin{definition}
Let $x \in \R^{n+1}$. We say that $x \in \d_*E$, the {\it measure theoretic boundary} of $E$, if 
$$
\limsup_{r \to 0} \frac{\mathcal L^{n+1} ( B(x,r) \cap E)}{r^{n+1}}>0
$$
and
$$
\limsup_{r \to 0} \frac{\mathcal L^{n+1} ( B(x,r) \setminus E)}{r^{n+1}}>0.
$$
\end{definition}

\begin{remark}
Note that $\d^* E \subset \d_* E$ and $\HH^n(\d_* E \setminus \d^* E)=0$ (see \cite[p.\ 208]{EvansGariepy}). Moreover, if $E$ has locally finite perimeter, then $\|\d E\|= \HH^n|_{\d^* E}$ (see \cite[p.\ 205]{EvansGariepy}).
\end{remark}
A useful criterion that allows us to determine whether a set has  locally finite perimeter, whose proof can be found in \cite[p.\ 222]{EvansGariepy}, is the following:
\begin{theorem}\label{thm:BV-Criterion}
If $E \subset \R^{n+1}$ is $\mathcal L^{n+1}$--measurable, then it has locally finite perimeter if and only if $\HH^n(K \cap \d_* E)<\infty$, for each compact set $K \subset \R^{n+1}$.
\end{theorem}
 We now state the {\it generalized Gauss-Green theorem}. For a proof see \cite[p. 209]{EvansGariepy}.
\begin{theorem}\label{thm:Gauss-Green}
Let $E \subset \R^{n+1}$ have locally finite perimeter. Then for each $x \in \d^* E$ there exists a unique measure theoretic unit outer normal $\nu_E(x)$ such that
\begin{equation}\label{eq:Gauss-Green}
\int_E \,\,\divv \phi\,\, d\mathcal L^{n+1} = \int_{\d_* E} (\phi \cdot \nu_E) \,d\mathcal \HH^{n},
\end{equation}
for all $\phi \in C^1_c(\R^{n+1};\R^{n+1})$.
\end{theorem}

\section{Riesz transform and rectifiability}

In this section we will state a theorem involving the relationship between Riesz transforms and rectifiability and derive a version of this which is better suited for our purposes. 

First we need to introduce some additional notation.
 Given a signed Radon measure $\nu$ in $\R^{n+1}$ we consider the $n$-dimensional Riesz
transform
$$\RR\nu(x) = \int \frac{x-y}{|x-y|^{n+1}}\,d\nu(y),$$
whenever the integral makes sense (for example, when $\nu$ has bounded support and $x\not\in \supp \nu$). For $\ve>0$, the $\ve$-truncated Riesz transform is given by 
$$\RR_\ve \nu(x) = \int_{|x-y|>\ve} \frac{x-y}{|x-y|^{n+1}}\,d\nu(y).$$
For $\delta\geq0$
 we set
$$\RR_{*,\delta} \nu(x)= \sup_{\ve>\delta} |\RR_\ve \nu(x)|.$$
In the case $\delta=0$ we write $\RR_{*} \nu(x):= \RR_{*,0} \nu(x)$.

If $\mu$ is a fixed Radon measure and $f\in L^1_{loc}(\mu)$, we also write
$$\RR_\mu f = \RR(f\mu),\quad \RR_{\mu,\ve} f = \RR_\ve(f\mu),\quad
\RR_{\mu,*,\delta} f= \RR_{*,\delta} (f\mu),\quad \RR_{\mu,*} f= \RR_{*} (f\mu),$$
whenever these notions make sense. We say that $\RR_\mu$ is bounded in $L^2(\mu)$ if the
operators $\RR_{\mu,\ve}$ are bounded in $L^2(\mu)$ uniformly on $\ve>0$.

Given a ball $B\subset\R^{n+1}$, we denote
$$\Theta_\mu(B) = \frac{\mu(B)}{r(B)^n},\qquad P_\mu(B) = \sum_{j\geq0} 2^{-j}\,\Theta_\mu(2^jB).$$
So $\Theta_\mu(B)$ is the $n$-dimensional density of $\mu$ on $B$ and $P_\mu(B)$ is some kind of smoothened version of this density.
For $f\in L^1_{loc}(\mu)$ and $A\subset \R^{n+1}$, we write
$$m_{\mu,A}(f) = \frac1{\mu(A)}\int_A f\,d\mu.$$
Given an $n$-plane $L\subset \R^{n+1}$, we also denote
$$\beta_{\mu,1}^L(B) = \frac1{r(B)^n}\int_B\frac{\dist(x,L)}{r(B)}\,d\mu(x).$$

The following theorem has been recently proved in \cite{GT}. This will be a fundamental tool for the proof of Theorem \ref{t:main}.

\begin{theorem}[{Girela-Sarri\'on, Tolsa}]\label{teo0}
Let $\mu$ be a Radon measure on $\R^{n+1}$ and $B\subset \R^{n+1}$ a ball so that the following conditions
hold:
\begin{itemize}

\item[(a)] For some constant $C_0>0$, $C_0^{-1}r(B)^n\leq \mu(B)\leq C_0\,r(B)^n$.

\item[(b)]  $P_\mu(B) \leq C_0$, and $\mu(B(x,r))\leq C_0\,r^n$ for all $x\in B$ and $0<r\leq r(B)$.

\item[(c)] There is some $n$-plane $L$ passing through the center of $B$ such that for some $0<\delta\ll 1$, it holds $\beta_{\mu,1}^L(B)\leq \delta$.

\item[(d)] $\RR_{\mu|_B}$ is bounded in $L^2(\mu|_{B})$ with $\|\RR_{\mu|_B}\|_{L^2(\mu|_{B})\to L^2(\mu|_{B})}\leq C_1$.

\item[(e)] For some constant $0<\tau\ll1$,
$$\int_{B} |\RR\mu(x) - m_{\mu,B}(\RR\mu)|^2\,d\mu(x) \leq \tau \,\mu(B).$$
\end{itemize}
Then there exists some constant $\theta>0$ such that if $\delta,\tau$ are small enough (depending on $C_0$ and $C_1$),
there is a uniformly $n$-rectifiable set $\Gamma\subset\R^{n+1}$ such that
$$\mu(B\cap \Gamma)\geq \theta\,\mu(B).$$
The UR constants of $\Gamma$ depend on all the constants above.
\end{theorem}
\vv

In the statement (e), $\RR\mu(x)$ should be understood in the principal value sense. That is,
$$\RR\mu(x) = \lim_{\ve\to0}\RR_\ve\mu(x).$$
The fact that $\RR_{\mu|_B}$ is bounded in $L^2(\mu|_{B})$ 
guaranties the existence of the principal value for $\mu$-a.e. $x\in B$.
This follows easily from the results of \cite{NToV-pubmat}, arguing as in 
\cite[Chapter 8]{Tolsa} with the Cauchy transform replaced by the Riesz transform.

\vv

Note that, in particular, a remarkable consequence of the theorem above is that a big piece of $\mu|_B$ is mutually absolutely continuous with respect to (a big piece of) $\HH^n|_\Gamma$.\vv

By applying Theorem \ref{teo0} to the normalized measure $\frac{r(B)^n}{\mu(B)}\mu$, we obtain the following.

\begin{corollary}\label{teo0'}
Let $\mu$ be a Radon measure on $\R^{n+1}$ and $B\subset \R^{n+1}$ a ball with $\mu(B)>0$ so that the following conditions
hold:
\begin{itemize}

\item[(a)]  For some constant $C_0>0$, $P_\mu(B) \leq C_0\,\Theta_\mu(B)$ and $\Theta_\mu(B(x,r))\leq C_0\,\Theta_\mu(B)$ for all $x\in B$ and $0<r\leq r(B)$.

\item[(b)] There exists some $n$-plane $L$ passing through the center of $B$ such that for some $0<\delta\ll 1$, it holds $\beta_{\mu,1}^L(B)\leq \delta\,\Theta_\mu(B)$.

\item[(c)] $\RR_{\mu|_B}$ is bounded in $L^2(\mu|_{B})$ with $\|\RR_{\mu|_B}\|_{L^2(\mu|_{B})\to L^2(\mu|_{B})}\leq C_1\,\Theta_\mu(B)$.

\item[(d)] For some constant $0<\tau\ll1$,
$$\int_{B} |\RR\mu(x) - m_{\mu,B}(\RR\mu)|^2\,d\mu(x) \leq \tau \,\,\Theta_\mu(B)^2\,\mu(B).$$
\end{itemize}
Then there exists some constant $\theta>0$ such that if $\delta,\tau$ are small enough (depending on $C_0$ and $C_1$),
there is a uniformly $n$-rectifiable set $\Gamma\subset\R^{n+1}$ such that
$$\mu(B\cap \Gamma)\geq \theta\,\mu(B).$$
The UR constants of $\Gamma$ depend on all the constants above.
\end{corollary}

For our purposes in connection with harmonic measure, the following variant of the preceding result will be more appropriate.

\begin{theorem} \label{teo1}
Let $\mu$ be a Radon measure in $\R^{n+1}$ and $B\subset \R^{n+1}$ a ball with $\mu(B)>0$ so that the following conditions
hold:
\begin{itemize}
\item[(a)] For some constant $C_0>0$, $P_\mu(B) \leq C_0\,\Theta_\mu(B)$.

\item[(b)] There is some $n$-plane $L$ passing through the center of $B$ such that, for some constant $0<\delta\ll 1$, $\beta_{\mu,1}^L(B)\leq \delta\,\Theta_\mu(B)$.

\item[(c)] For some constant $C_1>0$, there is $G_B\subset B$
such that
$$\sup_{0<r\leq 2 r(B)} \frac{\mu(B(x,r))}{r^n} + \RR_*(\chi_{2B}\,\mu)(x)\leq 
C_1\,\Theta_\mu(B)\quad \mbox{ for all $x\in G_B$}$$
and
$$\mu(B\setminus G_B)\leq \delta \,\mu(B).$$

\item[(d)] For some constant $0<\tau\ll1$,
$$\int_{G_B} |\RR\mu(x) - m_{\mu,G_B}(\RR\mu)|^2\,d\mu(x) \leq \tau \,\Theta_\mu(B)^2\mu(B).$$
\end{itemize}

Then there exists some constant $\theta>0$ such that if $\delta,\tau$ are small enough (depending on $C_0$ and $C_1$),
then there is a uniformly $n$-rectifiable set $\Gamma\subset\R^{n+1}$ such that
$$\mu(G_B\cap \Gamma)\geq \theta\,\mu(B).$$
The UR constants of $\Gamma$ depend on all the constants above.
\end{theorem}

\vv

\begin{remark}\label{rempv}
The condition that
$$\sup_{0<r\leq 2 r(B)} \frac{\mu(B(x,r))}{r^n} + \RR_*(\chi_{2B}\,\mu)(x)\leq 
C_1\,\Theta_\mu(B)$$
for every $x\in G_B$ given by (c) ensures that the principal value
$$\RR\mu(x) = \lim_{\ve\to0}\RR_\ve\mu(x)$$
exists for $\mu$-a.e.\ $x\in G_B$. This is due to the fact that the assumption (d) implies the $L^2(\mu|_{G_B})$ boundedness of 
$\RR_{\mu|_{G_B}}$. This is shown in the proof below.
\end{remark}
\vv

\begin{proof}[Proof of Theorem \ref{teo1}] 
We will show that the assumptions of Theorem \ref{teo0'} hold for 
$$\wt \mu = \mu|_{B^c\cup G_B}.$$
The assumptions (a) and (b) are clearly satisfied (because $\delta\ll1$) and thus we only have to check (c) and (d).

Suppose that the assumptions of Theorem \ref{teo1} hold. 
Let $\sigma=\mu|_{2B}$ and let $p_1, p_2>0$ be two big constants to be chosen momentarily. 
Denote
$$\mathcal M_n\sigma(x)= \sup_{r>0}\frac{\sigma(B(x,r))}{r^n}.$$
Let us also set 
$$E^1_{p_1} = \{ x \in \R^{n+1}: \mathcal M_n \sigma(x)> p_1 \,\Theta_\sigma(B)\}$$ and
$$E^2_{p_2} = \{ x \in \R^{n+1}: \RR_* \sigma(x)>p_2 \,\Theta_\sigma(B)\}.$$

 For $x\in E^1_{p_1}$, we denote 
$$\rho_1(x) = \sup\bigl\{r>0 : \sigma(B(x,r))> p_1\Theta_\sigma(B)\, r^n\}
$$
and for $x\in E^2_{p_2}$,
$$\rho_2(x) =\sup\bigl\{r >0 : |\RR_r\sigma(x)|
> p_2\Theta_\sigma(B)\bigr\}.$$
Define
$$H_i = \bigcup_{x\in E^i_{p_i}} B(x,\rho_i(x)),\,\, i=1,2.$$
Note that $H_1$ and $H_2$ are open sets and for $p_1$ and $p_2$ big enough it not hard to show that $2B\cap (H_1 \cup H_2)\subset 2B
\setminus G_B$. Indeed, it is clear that every ball $B_r$ with $\sigma(B_r) > p_1\,\Theta_\sigma(B)  r^n$ satisfies $B_r \subset H_1$.
Notice that if $y \in B\cap H_1$, then there is $x \in  E^1_{p_1}$ so that $y \in B(x,\rho_1(x))$, and so 
$$\sigma(B(y, 2\rho_1(x))) \ge \sigma(B(x, \rho_1(x))) \ge p_1\,\Theta_\sigma(B) \rho_1(x)^n =p_1\,\Theta_\sigma(B)  2^{-n}[2\rho_1(x)]^n.$$
We conclude that
$2B\cap H_1  \subset 2B \setminus G_B$, if we choose $p_1$ so that $p_1 > 2^n C_1$. 

We turn our attention to $H_2$. If  $y \in B\cap H_2 \setminus H_1$, then there exists $x \in E^2_{p_2}$ so that $y \in B(x, \rho_2(x) )$. We shall show that 
\begin{equation}\label{e:diff.Riesz}
|\RR_{\rho_2(x)} \sigma(x) -\RR_{\rho_2(x)} \sigma(y) | \leq C p_1 \Theta_\sigma(B),
\end{equation} 
where $C >0$ is some absolute constant depending only on the dimension. Indeed, we have that  
\begin{align*}
|\RR_{\rho_2(x)} &\sigma(x) -\RR_{\rho_2(x)} \sigma(y) | \\
\leq &|\RR_{\rho_2(x)} ( \chi_{ B(y, 2 \rho_2(x) )}\sigma )(x) | + |\RR_{\rho_2(x)} ( \chi_{B(y, 2 \rho_2(x))}\sigma) (y)| \\
&+|\RR_{\rho_2(x)}( \chi_{\R^{n+1} \setminus B(y, 2 \rho_2(x))} \sigma)(x) -\RR_{\rho_2(x)} (\chi_{\R^{n+1} \setminus B(y, 2 \rho_2(x) )}\sigma)(y)| \\
& =:I_1+I_2+I_3.
\end{align*}
Notice now that 
$$ I_1+I_2 \leq C_n \frac{\sigma(B(y, 2 \rho_2(x)))}{\rho_2(x)^n} \leq  2^n p_1 \Theta_\sigma(2B)\leq  C\, p_1 \Theta_\sigma(B),$$
where the second inequality follows form the fact that $y \not \in H_1$. It just remains to handle $I_3$. To this end, 
\begin{align*}
I_3 &= |\RR( \chi_{\R^{n+1} \setminus B(y, 2 \rho_2(x))} \sigma)(x) -\RR(\chi_{\R^{n+1} \setminus B(y, 2 \rho_2(x) )}\sigma)(y)|\\
&\leq \wt C_n \int_{\R^{n+1} \setminus B(y, 2 \rho_2(x))} \frac{|x-y|}{|z-y|^{n+1}}\,d \sigma (z) \\
&\leq \wt C_n \sum_{j \geq 1} \frac{\rho_2(x)}{(2^j \rho_2(x))^{n+1}} \,\sigma(B(y, 2^{j+1}\rho_2(x))\\&\leq \wt C_n\, 2^n p_1 \Theta_\sigma(B),
\end{align*}
where in the last inequality we used that $y \not \in H_1$. This concludes the proof of \eqref{e:diff.Riesz}. Therefore, 
since $|\RR_{\rho_2(x)} \sigma(x)| > p_2  \Theta_\sigma(B)$, we have that $2B\cap H_2 \setminus H_1 \subset 
2B \setminus G_B$, if we choose $p_2$ so that $p_2 -C\, 2^n p_1> C_1$.

Let $H=H_1 \cup H_2$ and consider the $1$-Lipschitz function
$$\Phi(x) = \dist(x,H^c) \geq \max(\rho_1(x), \rho_2(x)),$$
and the associated ``suppressed kernel''
$$K_\Phi(x,y) = \frac{x-y}{\bigl(|x-y|^2 + \Phi(x)\,\Phi(y)\bigr)^{(n+1)/2}}.$$
We consider the operator $\RR_{\Phi,\sigma}$ defined by
$$\RR_{\Phi,\sigma}f(x) = \int K_\Phi(x,y)\,f(y)\,d\sigma(y),$$
and its $\ve$-truncated version (for $\ve>0$)
$$\RR_{\Phi,\ve,\sigma}f(x) = \int_{|x-y|>\ve} K_\Phi(x,y)\,f(y)\,d\sigma(y).$$
We also set
$$\RR_{\Phi,*,\sigma}f(x) = \sup_{\ve>0} \RR_{\Phi,\ve,\sigma}f(x).$$
We say that $\RR_{\Phi,\sigma}$ is bounded in $L^2(\sigma)$ if the operators $\RR_{\Phi,\ve,\sigma}$
are bounded in $L^2(\sigma)$ uniformly on $\ve>0$.

We now prove that
\begin{equation}\label{eq:suppRieszbound}
\RR_{\Phi,*,\sigma} 1(x) \leq C(p_1,p_2)\,\Theta_\sigma(B),
\end{equation}
 for all $x\in\R^{n+1}$. To do so, we need the following lemma which proof can be found in \cite[Lemma 5.5]{Tolsa}.
\begin{lemma}\label{lem:Tolsa-suppress}
 Let $x \in \R^{n+1}$ and $r_0 \geq 0$ so that $\sigma(B(x,r)) \leq A_1 r^n$ for $r \geq r_0$ and $|\RR_\ve \sigma (x)| \leq A_2$ for $\ve \geq r_0$. If $\Phi(x) \geq r_0$, then there exists $	C>0$, so that $|\RR_{\Phi,\ve, \sigma} 1(x)| \leq C\,A_1+  A_2$ for all $\ve>0$. 
 \end{lemma}
By Lemma \ref{lem:Tolsa-suppress} for $A_1=p_1 \Theta_\sigma(B)$, $A_2= p_2\Theta_\sigma(B)$ and $r_0=\max\{\rho_1(x), \rho_2(x)\}$, we obtain \eqref{eq:suppRieszbound}. We further apply the $Tb$ theorem for suppressed operators by Nazarov--Treil--Volberg \cite{NTV14-prep} (see also Corollary 5.33 in \cite{Tolsa}) and it follows that
$\RR_{\Phi,\sigma}:L^2(\sigma) \to L^2(\sigma)$ is bounded  with norm 
$$\|\RR_{\Phi,\sigma}\|_{L^2(\sigma)\to L^2(\sigma)}\lesssim\Theta_\sigma(B)= \Theta_\mu(B).$$
Since $\Phi$ vanishes on $G_B\subset H^c$, we have that $\RR_{\wt\mu|_B}:L^2(\wt \mu|_{B}) \to L^2(\wt \mu|_{B})$ is bounded and $\|\RR_{\wt\mu|_B}\|_{L^2(\wt\mu|_B)\to L^2(\wt\mu|_B)}\lesssim\Theta_\wt\mu(B)$.

To check that the condition (e) in Theorem \ref{teo0} holds, we write
\begin{align*}
\int_{B} |\RR\wt\mu(x) - m_{\wt\mu,B}(\RR\wt \mu)|^2\,d\wt\mu(x) & \lesssim
\int_{G_B} |\RR\mu(x) - m_{\mu,G_B}(\RR\mu)|^2\,d\mu(x) \\
&\quad +
\int_{G_B} |\RR (\mu-\wt \mu)|^2\,d\mu \\
& = I_1+ I_2.
\end{align*}
Concerning $I_1$, by assumption we have
$$I_1\leq \tau\,\Theta_\mu(B)^2\,\mu(B) \approx \tau\,\Theta_{\wt\mu}(B)^2\,\wt\mu(B).$$
For $I_2$, notice that $\mu-\wt \mu=\mu|_{B\setminus G_B}$ and, further, recall that $\Phi$ vanishes 
on $G_B$ because $G_B\subset H^c$ and so $\RR (\mu-\wt \mu)(x)=\RR_\Phi (\mu-\wt \mu) (x)$ for all
$x\in G_B$. Further, $\RR_{\Phi,\sigma}$ is bounded in $L^4(\sigma)$, by using the
boundedness of $\RR_{\Phi,\sigma}$
 from $L^{1}(\sigma)$ to $L^{1,\infty}(\sigma)$ (see Lemma 5.27 of \cite{Tolsa}, for example) and duality. So we have
\begin{align*}
\int_{G_B} |\RR (\mu-\wt \mu)|^2\,d\mu & \leq \mu(G_B)^{1/2}\,\|\RR_\Phi (\mu|_{B\setminus G_B})\|_{
L^4(\mu|_{G_B})}^2 \\
& \lesssim \Theta_{\sigma}(B)^2 \,\mu(G_B)^{1/2}\,\mu(B\setminus G_B)^{1/2}\lesssim
\delta^{1/2}\,\Theta_{\wt\mu}(B)^2 \,\wt\mu(B).
\end{align*}
Gathering the estimates obtained for $I_1$ and $I_2$ we get
$$\int_{B} \bigl|\RR\wt\mu(x) - m_{\wt\mu,B}(\RR\wt \mu)\bigr|^2\,d\wt\mu(x)\lesssim
(\tau + \delta^{1/2})\,\Theta_{\wt\mu}(B)^2 \,\wt\mu(B),$$
which shows that the assumption (d) of Theorem \ref{teo0'} holds.
\end{proof}

\vv

\section{Background on Harmonic Measure}\label{s:harmeasure}

Let us first recall some definitions and basic facts concerning harmonic measure and Green functions.

\subsection{Harmonic measure and Green function} For a (possibly unbounded) domain $\Omega \subset \R^{n+1}$ and $x \in \Omega$, one can construct the harmonic measure $\omega^x_\Omega$ (see e.g. \cite[p. 172]{ArmGard} or \cite[p. 217]{Hel}). In fact, for any continuous function $f$, the Perron solution  for the boundary function $f$ is given by 
$$ H_f(x) = \int_{\partial_\infty \Omega} f(y) \,d\omega_\Omega^x(y),$$
where $\d_\infty\Omega=\d\Omega$ if $\Omega$ is bounded and $\d_\infty\Omega=\d\Omega\cup\{\infty\}$ otherwise.
 Remark  that constant functions are continuous and  since $H_1(x)=1$, for any $x \in \Omega$, we have that $\omega_\Omega^x(\partial_\infty \Omega)=1$, for any $x \in \Omega$.

Let $\EE$ denote the fundamental solution for the Laplace equation in $\R^{n+1}$, so that $\mathcal{E}(x)=c_n\,|x|^{1-n}$ for $n\geq 2$, $c_n>0$. A {\it Green function} $G_{\Omega}:\Omega\times \Omega\rightarrow[0,\infty]$ for an open set $\Omega\subset \R^{n+1}$ is a function with the following properties: for each $x\in \Omega$, $G_{\Omega}(x,y)=\EE(x-y)+h_{x}(y)$ where $h_{x}$ is harmonic on $\Omega$, and whenever $v_{x}$ is a nonnegative superharmonic function that is the sum of $\EE(x-\cdot)$ and another superharmonic function, then  $v_{x}\geq G_{\Omega}(x,\cdot)$ (\cite[Definition 4.2.3]{Hel}). 

An open subset of $\R^{n+1}$ having a Green function is called a {\it Greenian} set. By \cite[Theorem 4.2.10]{Hel}, all open subsets of $\R^{n+1}$ are Greenian for $n\geq 2$. Moreover, Green function can be written as follows (see \cite[Lemma 6.8.1]{ArmGard}): for $x,y\in\Omega$, $x\neq y$, define
\begin{equation}\label{green}
G_\Omega(x,y) = \mathcal{E}(x-y) - \int_{\partial\Omega} \mathcal{E}(x-z)\,d\omega^y(z).
\end{equation}

For  $x\in\R^{n+1}\setminus \Omega$ and $y\in\Omega$, we will also set 
\begin{equation}\label{green2}
G_\Omega(x,y)=0.
\end{equation}

The kernel of the Riesz transform is
\begin{equation}\label{eqker}
K(x) = c_n\,\nabla \EE(x),
\end{equation}
for a suitable absolute constant $c_n$. 
For $x\in\R^{n+1}\setminus \overline\Omega$, since $K(x-\cdot)$ is harmonic in $\Omega$, we have
\begin{equation}\label{eqclau2}
\RR\omega^y(x) = \int K(x-z)\,d\omega^y(z) = K(x-y).
\end{equation}
For $x\in\Omega$, by \rf{eqker} and  \rf{green}  we get
\begin{align}\label{eqclau1}
\RR\omega^y(x) = 
c_n\nabla_x\int \EE(x-z)\,d\omega^y(z) & = c_n\,\nabla_x\bigl(\EE(x-z) - G_{\Omega}(x,y)\bigr) \nonumber\\
&= K(x-y) - c_n\,\nabla_x G_{\Omega}(x,y).
\end{align}

\vv
The following result is also standard. For the proof of the precise statements, see 
\cite{AHMMMTV15}, for example.

\begin{lemma}\label{l:w>G}
Let $n\ge 2$ and $\Omega\subset\R^{n+1}$ be a domain.
Let $B=B(x_0,r)$ be a closed ball with $x_0\in\pom$ and $0<r<\diam(\pom)$. Then, for all $a>0$,
\begin{equation}
 \omega_{\Omega}^{x}(aB)\gtrsim \inf_{z\in 2B\cap \Omega} \omega_{\Omega}^{z}(aB)\, r^{n-1}\, G_{\Omega}(x,y)\quad\mbox{
 for all $x\in \Omega\backslash  2B$ and $y\in B\cap\Omega$,}
 \end{equation}
 with the implicit constant independent of $a$. 
\end{lemma}

The above lemma was originally stated in \cite{AHMMMTV15} for bounded domains, but it holds for unbounded domains with the same proof using the fact that, for $n\geq 2$, any domain $\Omega\subset \bR^{n+1}$ is Greenian and, if it is unbounded, $\infty$ is a Wiener regular point (see \cite[Theorem 6.7.1]{ArmGard}). 

\subsection{$\Delta$-regular domains}

\begin{definition} \label{defdelta}
A domain $\Omega\subsetneq \bR^{n+1}$ is {\it $(\beta,R)$-$\Delta$-regular} if there are 
$R,\beta>0$ so that 
\begin{equation}
\label{e:dregular}
 \sup_{\xi\in \d\Omega} \,\sup_{x\in \d B(\xi,r/2)\cap\Omega} \omega_{B\cap \Omega}^{x}(\d B(\xi,r)\cap \Omega)  \leq\beta<1 \mbox{ for }r\in (0,R).
 \end{equation}
We call a domain $\Omega$ {\it two-sided $\Delta$-regular} if $\ext(\Omega):=(\cnj{\Omega})^{c}$ is also a $\Delta$-regular domain. 
\end{definition}

If we want to specify the constants $\beta,R$ above, we will talk about $(\beta,R)$-$\Delta$-regularity. 
It can be shown that one obtains an equivalent definition if the second supremum above is taken over $x\in \d B(\xi,\delta r)\cap\Omega$,
for any fixed constant $0<\delta<1$.

\begin{definition}
 Let $n\geq 2$ and let $\Cap$ denote the Newtonian capacity. A domain $\Omega\subset \R^{n+1}$ satisfies the {\it capacity density condition} (or CDC) if there is $R_{\Omega}>0$ and $c_\Omega>0$ so that $\Cap(B\backslash \Omega)\geq c_\Omega \,r(B)^{n-1}$ for any ball $B$ centered on $\d\Omega$ of radius $r(B)\in (0,R_{\Omega})$. 
\label{def:cdc}
\end{definition}

Although this result will be not used in this paper, we recall that
the CDC is equivalent to $\Delta$-regularity for $n\geq 2$:

\begin{theorem}  \cite[Lemma 3]{Anc86} For $n\geq 2$, if  $\Omega\subset\bR^{n+1}$ and $B$ is centered on $\d\Omega$, then $\Cap(B\backslash \Omega)\gtrsim r(B)^{n-1}$ if and only if there is $\beta\in (0,1)$ so that $\omega_{B\cap \Omega}^{x}(\d B\cap \Omega)\leq \beta$ on $\d (\frac12B)\cap \Omega$. In particular, $\Omega$ is $\Delta$-regular if and only if it satisfies the CDC. 
\label{l:ancona}
\end{theorem}

Below we recall some estimates that are written in more generality but will be applied in the setting of $\Delta$-regular domains. The following result is well known and follows by standard techniques, see for example \cite[Lemma 2.3]{AM15}.
\vv

\begin{lemma}
\label{l:holder}
Let $\Omega\subset \bR^{n+1}$, $\delta\in (0,1)$, $\xi\in \d\Omega$ and suppose that 
\[\omega_{B(\xi,r)\cap \Omega}^{x}(\d B(\xi,r)\cap \Omega)  \leq \beta<1\; \mbox{ for }\; x\in  \d B(\xi,\delta r)\cap \Omega \;\mbox{ and }r\in (0,R).\] Then there is $\alpha=\alpha(\beta,\delta,n)$ so that for all $r\in (0,R)$
\begin{equation}\label{e:wholder}
 \omega_{\Omega}^{x}({B}(\xi,r)^{c})\lesssim_{\beta,\delta} \ps{\frac{|x-\xi|}{r}}^{\alpha} \mbox{ for } x\in \Omega\cap B(\xi,r).
 \end{equation}
In particular, $\xi$ is a regular point for $\d\Omega$. 
\end{lemma}

\vv
By the maximum principle, this implies the following.

\begin{corollary}\label{c:holder}
Let $\Omega\subset  \bR^{n+1}$, $\delta\in (0,1)$, $\xi\in \d\Omega$ and suppose that $\omega_{B(\xi,r)\cap \Omega}^{x}(\d B(\xi,r)\cap \Omega)  \leq \beta<1$ for $x\in  \d B(\xi,\delta r)\cap \Omega$ and $r\in (0,R)$. Let $u$ be a nonnegative function which is continous in
$\overline{B(\xi,\delta r)\cap \Omega}$ and harmonic in $B(\xi,\delta r)\cap \Omega$, and vanishes continuously on $\cnj{B(\xi,r)}\cap\d\Omega$. Then there is $\alpha=\alpha(\beta,\delta,n)$ so that for all $r\in (0,R)$,
\begin{equation}
u(x)\lesssim_{\beta,\delta} \biggl(\,\sup_{B(\xi,r)\cap \Omega} u\biggr) \ps{\frac{|x-\xi|}{r}}^{\alpha} \mbox{ for } x\in \Omega\cap B(\xi,r).
\label{e:holder}
 \end{equation}
\end{corollary}

\vv

\begin{lemma}
\label{l:bourgain}
Let $\Omega\subset \bR^{n+1}$ be a $(\beta,R)$-$\Delta$-regular domain, for $\beta\in (0,1)$, $R>0$. Then there are $\delta_0\in(0,1)$ and  $\rho>0$, both depending on $\beta,n$, so that for all $r\in (0,R)$ and 
$\xi\in \d\Omega$, 
\begin{equation}
 \omega_{\Omega}^{x}(B(\xi,r))\geq 1/2 \;\; \mbox{  for all }x\in B(\xi,\delta_0r)\cap \Omega.
 \label{e:bourgain}
 \end{equation}
\end{lemma}

\begin{proof} By Lemma \ref{l:holder},
if $|x-\xi|<\delta_0$ for some positive $\delta_0$ small enough depending on $\beta,n$, then
$ \omega_\Omega^{x}(\d_\infty\Omega\setminus B(\xi,r)^{c})\leq \frac12$ and thus
$ \omega_\Omega^{x}({B}(\xi,r))\geq \frac12$.
\end{proof}

If $\Omega$ is $\Delta$-regular, then by \eqn{bourgain} and  \Lemma{w>G}, we have
 \begin{equation}
 \label{e:w>g}
 \omega_{\Omega}^{x}(2\delta_0^{-1}B)\gtrsim  r^{n-1}\, G_{\Omega}(x,y)\quad\mbox{
 for all $x\in \Omega\backslash  2B$ and $y\in B\cap\Omega$, $0<r(B)<\dfrac{\delta_0R}2$.}
 \end{equation}


\vv

\subsection{Admissible domains and relevant estimates}

\begin{definition}\label{def:adm}
A domain $\Omega \subset \R^{n+1}$ is {\it admissible} if
\begin{enumerate}
\item $\Omega^+=\Omega$ and $\Omega^-=\ext(\Omega)$ are Wiener regular;
\item  $\partial \Omega^+=\partial \Omega^-=\partial \Omega$;
\item There exist $x^\pm \in \Omega^\pm$ such that if 

\begin{equation}
    u(x)=
    \begin{cases}
      G_{\Omega^{+}}(x,x^+), & \text{\text for}\,\, x \in \Omega^+, \\
      -G_{\Omega^{-}}(x,x^-), & \text{for}\,\, x \in \Omega^-,
    \end{cases}
\end{equation}
where $G_{\Omega^{\pm}}(x,x^\pm)$ is the Green function in $\Omega^\pm$ with pole at $x^\pm$,
and $\delta(x) :=\dist(x, \partial \Omega)$, then for every $\xi \in \partial \Omega$ there exists 
$R>0$ with $R < \min\{\delta(x^+), \delta(x^-)\}$ so that $u \in C(B(\xi,R)) \cap W^{1,2}(B(\xi, R))$. 
\end{enumerate}
\end{definition}

\vv

\begin{theorem} \label{t:ACF} \cite[Lemma 5.1]{ACF84} Let $\Omega \subset \R^{n+1}$ be an admissible domain and $u^{\pm}=G_{\Omega^{\pm}}(\cdot,x^{\pm})$. Then for $x\in \d\Omega$ there is $0<R<\min\{\dist(x^{+},\d\Omega),\dist(x^{-},\d\Omega)\}$ such that the quantity
\begin{equation}\label{eq*123}
\gamma(x,r) = \left(\frac{1}{r^{2}} \int_{B(x,r)} \frac{|\grad u^+(y)|^{2}}{|y-x|^{n-1}}dy\right)\cdot \left(\frac{1}{r^{2}} \int_{B(x,r)} \frac{|\grad u^{-}(y)|^{2}}{|y-x|^{n-1}}dy\right)
\end{equation}
is a non-decreasing function of $r\in (0,R)$ and $\gamma(x,R)<\infty$, that is,
\begin{equation}\label{e:gamma}
\gamma(x,r_{1})\leq \gamma(x,r_{2})<\infty \;\;  \mbox{ for } \;\; 0<r_{1}\leq r_{2}<R.
\end{equation}
\end{theorem}

\vv

\begin{lemma}\label{l:beurling}
\cite[Theorem 3.3]{KPT09}
Let $\Omega\subset \R^{n+1}$ be an admissible domain and let $\omega^{\pm}=\omega_{\Omega^{\pm}}^{x_{\pm}}$. Let $0<R<\min\{\dist(x^{+},\d\Omega),\dist(x^{-},\d\Omega)\}$ be as in \Theorem{ACF}. Then for $0<r<R/4$ and $\xi\in \d\Omega$,
\begin{equation}\label{e:beurling0}
\frac{\omega^{\pm}(B(\xi,r))}{r^{n}}\lesssim  \left(\frac{1}{r^{2}} \int_{B(\xi,2r)} \frac{|\grad u^{\pm}(y)|^{2}}{|y-\xi|^{n-1}}dy\right)^{\frac{1}{2}} 
\lesssim \left(\frac{1}{r^{n+3}} \int_{B(\xi,4r)}(u^{\pm})^{2}\right)^{\frac{1}{2}}
\end{equation}
and in particular,
\begin{equation}\label{e:beurling}
\frac{\omega^{+}(B(\xi,r))}{r^{n}}\,\frac{\omega^{-}(B(\xi,r))}{r^{n}}\lec \gamma(\xi,2r)^{\frac{1}{2}},
\end{equation}
where $\gamma(\xi,2r)$ is defined by \rf{eq*123}.
\end{lemma}
\vv

\begin{lemma}
\label{l:otherside}
 Let $\Omega^+=\Omega\subset \R^{n+1}$ and $\Omega^-={\rm ext}( \Omega)$ be $\Delta$-regular domains. If $$0<R<\min\{\dist(x^{+},\d\Omega^{+}),\dist(x^{-},\d\Omega^{-})\},$$ then for $\xi\in \d\Omega$ and $r< \delta_0 R/4$,
\begin{equation}\label{eq:avgGreen-harm}
\left(\frac{1}{r^{n+1}} \int_{B(\xi,r) \cap \Omega^\pm}|\nabla u^{\pm}|^{2}\right)^{\frac{1}{2}}  \lesssim \left(\frac{1}{r^{n+3}} \int_{B(\xi, 2r)\cap \Omega^\pm}(u^{\pm})^{2}\right)^{\frac{1}{2}} \lesssim \frac{\omega^{\pm}(B(\xi,4 \delta_0^{-1}r))}{r^{n}}.
\end{equation}
In particular, 
\begin{equation}\label{e:otherside}
\gamma(\xi,r)^{\frac{1}{2}}\lec \frac{\omega^{+}(B(\xi,4\delta_0^{-1}r))}{r^{n}}
\frac{\omega^{-}(B(\xi,4\delta_0^{-1}r))}{r^{n}},
\end{equation}
where $\gamma(\xi,r)$ is defined by \rf{eq*123}.
\end{lemma}

\begin{proof}
We shall only deal with $\Omega^+$ since the result for $\Omega^-$ is identical.
Let $\xi\in \d\Omega^+$ and $4r< \delta_0 R$. Since $u^+$ vanishes continuously at the boundary of $\d \Omega^+$, we may extend it by zero in $\R^{n+1} \setminus \Omega^+$. Then, as the extended function (which we still denote it $u^+$) is non-negative and subharmonic in $\R^{n+1}$, by Caccioppoli's inequality, we infer 
$$
\left( \int_{B(\xi,r)}|\nabla u^+|^{2}\right)^{\frac{1}{2}}  \lesssim \left(\frac{1}{r^2} \int_{B(\xi, 2r)}(u^+)^{2}\right)^{\frac{1}{2}} \lesssim \omega^+(B(\xi,4 \delta_0^{-1}r)) \,r^{\frac{1-n}{2}},
$$
where the second inequality follows from \eqn{w>g}. 
 This shows \eqref{eq:avgGreen-harm}, which in turn implies \eqref{e:otherside}.
\end{proof}

\vv
\begin{lemma}\label{l:CDCadmissible}
If $\Omega^{+}=\Omega\subset \R^{n+1}$ and $\Omega^{-}$ are  $R_1$-$\Delta$-regular and $\partial \Omega^+=\partial \Omega^-$, then they are admissible domains.
\end{lemma}

\begin{proof}
Fix $x^\pm \in \Omega^\pm$ so that $\delta(x^\pm) >0$. The first two conditions of Definition \ref{def:adm} readily follow from our hypotheses. Fix now $\xi \in \partial \Omega$ and choose \[R< \min\{R_1, \delta(x^+), \delta(x^-)\}/ 4.\] 
Since $x^\pm \in \Omega^\pm \setminus B(\xi, 4 R)$, we have that $u^\pm$ is harmonic in $B(\xi, 2r) \cap \Omega^\pm$. Moreover, the common boundary $\partial \Omega$ is Wiener regular for $\Omega^\pm$ by \Lemma{holder}, which implies that $u^\pm$ vanishes continuously on $\partial \Omega$.  So $u \in C(B(\xi, R))$ and, by Lemma \ref{l:otherside}, it holds that $u \in W^{1,2}(B(\xi, 2R)\cap\Omega^\pm)$. Further, since $u^\pm$ is continuous in $B(\xi, 2R)\cap\Omega^\pm$ and vanishes on $\partial \Omega$, if we consider a function $\vphi\in C^\infty$ which equals $1$ on $B(\xi, R)$ and vanishes our of
$B(\xi, 2R)$, by standard arguments it turns out that $\vphi\,u^\pm\in W^{1,2}_0(B(\xi, 2R)\cap\Omega^\pm)$
(see \cite[Theorem 9.17]{Brezis}, for example). Hence, $\vphi\,u \in W^{1,2}_0(B(\xi, 2R))$,
and so $u \in W^{1,2}(B(\xi, R))$. This concludes our proof.
\end{proof}

\vv

\vv
\section{Blowups at points of mutual absolute continuity}

Given a set $G\subset \R^{n+1}$ and a ball $B\subset\R^{n+1}$, we denote
$$\beta_{G,\infty}(B)=  \inf_L \sup_{x\in G\cap B} \frac{\dist(x,L)}{r(B)},$$
where the infimum is taken over all $n$-planes $L$.
Also, we set
$$\wt\beta_{G,\infty}(B)=  \inf_S \frac{d_H(G\cap B,S\cap B)}{r(B)},$$
where $d_H$ stands for the Hausdorff distance and the infimum is taken over all half-spaces $S\subset\R^{n+1}$ whose boundary contains the center of $B$.
To shorten notation, for a ball $B(x,r)$, we also write
 $\beta_{G,\infty}(x,r)$ and $\wt\beta_{G,\infty}(x,r)$ intead of  $\beta_{G,\infty}(B(x,r))$ and $\wt\beta_{G,\infty}(B(x,r))$, respectively.
\vv

This entire section is devoted to proving the following theorem.

\begin{theorem}\label{t:alpha0}
Let $\Omega^+\subset \bR^{n+1}$ and $\Omega^- = (\cnj{\Omega^+})^c$ be two $\Delta$-regular domains, so that $\partial\Omega^+=\partial\Omega^-$. Let $\omega^\pm$ be the harmonic measures for $\Omega^\pm$ with poles $x^\pm\in \Omega^\pm$, and $u^\pm=G_{\Omega^\pm}(x^\pm,\cdot)$. Suppose there is $E\subset \d\Omega^+$ such that $\omega^+|_{E}\ll\omega^-|_{E}\ll \omega^+|_{E}$ and $\omega^+(E)>0$.  Then, for $\omega^+$-a.e. $\xi\in E$,
\[
\lim_{r\rightarrow 0}\beta_{\d\Omega^+,\infty}(\xi,r)=0 \quad \mbox{ and } \quad
\lim_{r\rightarrow 0}\wt\beta_{\overline{\Omega^+},\infty}(\xi,r)=0.
\]
\label{t:alpha0}
\end{theorem}
\vv

\subsection{Tangent Measures}

For $a\in\R^{n+1}$ and $r>0$, we consider the map
$$T_{a,r}(x) = \frac{x-a}{r}.$$
Note that $T_{a,r}(B(a,r))=B(0,1)$. Recall also that, given a Radon measure $\mu$, the notation $T_{a,r}[\mu]$ stands for the image measure of $\mu$ by $T_{a,r}$.
That is,
$$T_{a,r}[\mu](A) = \mu(rA+a),\qquad A\subset\R^{n+1}.$$

\begin{definition}  
Let $\mu$ be a Radon measure in $\R^{n+1}$. We say that $\nu$ is a {\it tangent measure} of $\mu$ at a point $a\in\R^{n+1}$ if
$\nu$ is a non-zero Radon measure on $\R^{n+1}$ and there are sequences $\{r_i\}_i$ and $\{c_i\}_i$ of positive numbers, with $r_i\to0$, so that $c_i\,T_{a,r_{i}}[\mu]$ converges weakly to $\nu$ as $i\to\infty$.
\end{definition}

\begin{definition}  
Given two Radon measure $\mu$ and $\sigma$, we set
$$F_{B}(\mu,\sigma) = \sup_f \int f\,d(\mu-\sigma),$$
where the supremum is taken over all the $1$-Lipschitz functions supported on $B$. 
For $r>0$, we write
\[
F_{r}(\mu,\nu)= F_{\overline B(0,r)},\qquad
F_{r}(\mu)=F_{r}(\mu,0)=\int (r-|z|)_{+} d\mu.\]
\end{definition}


\vv
\begin{lemma}\cite[Proposition 1.11]{Pr87}
Let $\{\mu_i\}$ be a sequence of Radon measures such that $\limsup \mu_{i}(B(0,r))<\infty$ for all $r>0$. Then $\mu_{i}$ converges weakly to a measure $\mu$ if and only if $F_{r}(\mu_{i},\mu)\rightarrow 0$ for every $r>0$.
\end{lemma}

\vv
\begin{definition} \cite[Section 2]{Pr87} 
\begin{enumerate}[(a)]
\item A set $\cM$ of non-zero Radon measures in $\bR^{n+1}$ is a {\it cone} if $c\mu\in \cM$ whenever $\mu\in \cM$ and $c>0$.
\item A cone $\cM$ is a {\it d-cone} if $T_{0,r}[\mu]\in \cM$ for all $\mu\in \cM$ and $r>0$.
\item For a $d$-cone $\cM$, $r>0$, and $\mu$ a Radon measure with $0<F_{r}(\mu)<\infty$, we define the {\it distance} between $\mu$ and $\cM$ as 
\[
d_{r}(\mu,\cM)=\inf\ck{F_{r}\ps{\frac{\mu}{F_{r}(\mu)},\nu}: \nu\in \cM, F_{r}(\nu)=1 }
\]
\end{enumerate}
\end{definition}
\vv

For example, the set of measures 
\begin{equation}\label{eqflat}
\cF = \bigl\{c\,\HH^n|_L: c>0,\,\mbox{$L$ is an $n$-plane in $\R^{n+1}$ through the origin}\bigr\}
\end{equation}
is a d-cone.

The only fact about distances to cones that we will require later is the following equality, see \cite[Remark 2.8]{KPT09}: for any Radon measure $\mu$, $d$-cone $\cM$, and $r>0$,

\begin{equation}\label{e:dr}
d_{r}(\mu,\cM)=d_{1}(T_{0,r}[\mu],\cM)
\end{equation}

\vv

%

\vv


\begin{theorem}\label{t:tme} \cite[Theorem 2.5]{Pr87} If $\mu$ is a Radon measure on $\bR^{n+1}$, then $\Tan(\mu,x)\neq\varnothing$ for $\mu$-almost every $x\in \bR^{n+1}$. 
\end{theorem}



\vv

\begin{theorem} \label{t:ttt}
\cite[Theorem 14.16]{Mattila}
Let $\mu$ be a Radon measure on $\bR^{n+1}$. For $\mu$-almost every $x\in \bR^{n+1}$, if $\nu\in \Tan(\mu,x)$, the following hold:
\begin{enumerate}
\item $T_{y,r}[\nu]\in \Tan(\mu,x)$ for all $y\in \supp \nu$ and $r>0$.
\item $\Tan(\nu,y)\subset \Tan(\mu,x)$ for all $y\in \supp \nu$.
\end{enumerate}
\end{theorem}
\vv

\subsection{The Proof of \Theorem{alpha0}}

Assume the conditions of \Theorem{alpha0}. Set 
\[E^{*}=\ck{\xi\in E: \lim_{r\rightarrow 0} \frac{\omega^+(E\cap B(\xi,r))}{\omega^+(B(\xi,r))}= \lim_{r\rightarrow 0} \frac{\omega^-(E\cap B(\xi,r))}{\omega^-(B(\xi,r))}=1}.\]
By \cite[Corollary 2.14 (1)]{Mattila} and because $\omega^{+}$ and $\omega^{-}$ are mutually absolutely continuous on $E$,
\[\omega^+(E\backslash E^{*})= \omega^-(E\backslash E^{*})=0.\]
Also, set 
\begin{multline*}
\Lambda_{1} =\left\{\xi\in E^*\!\!: 0<h(\xi):=\frac{d\omega^-}{d\omega^+}(\xi)=\lim_{r\rightarrow 0} \frac{\omega^-( B(\xi,r))}{\omega^+( B(\xi,r))} \right. \\
 =\left. \lim_{r\rightarrow 0} \frac{\omega^-(E\cap B(\xi,r))}{\omega^+(E\cap B(\xi,r))}<\infty\right\}\end{multline*}
and 
\[
\Gamma = \ck{\xi\in \Lambda_{1}: \xi \mbox{ is a Lebesgue point for $h$ with respect to }\omega^+}.
\]
Again, by Lebesgue differentiation for measures (see \cite[Corollary 2.14 (2) and Remark 2.15 (3)]{Mattila}), $\Gamma$ has full measure in $E^{*}$ and hence in $E$.
\vv

%

The following is essentially taken from \cite{KPT09}, but we adjust it slightly so that we don't need to assume any doubling properties of harmonic measure.

\begin{lemma}\label{l:samew}
Let $\xi\in \Gamma$, $c_{j}\geq 0$,  and $r_{j}\rightarrow 0$ be such that $\omega_{j}^+=c_{j}T_{\xi,r_{j}}[\omega^+]\rightarrow \omega_{\infty}^+$. Then $\omega_{j}^-=c_{j}T_{\xi,r_{j}}[\omega^-]\rightarrow h(\xi)\omega_{\infty}^+$.
\end{lemma}

\begin{proof}
Let $\phi\in C_{c}(\bR^{n+1})$ have support in $B(0,M)$ for some $M>0$. Let $\phi_{\xi,r_{j}}=\phi\circ T_{\xi,r_{j}}$. Then
\begin{align*}
\lim_{j\rightarrow\infty}  \int \phi\,d\omega^-_j & =
\lim_{j\rightarrow\infty}  c_{j} \int \phi_{\xi,r_{j}}d\omega^- \\
&=\lim_{j\rightarrow\infty}  c_{j}\int_E \phi_{\xi,r_{j}}d\omega^-+\lim_{j\rightarrow\infty} c_{j}\int_{\d\Omega^-\backslash E} \phi_{\xi,r_{j}}d\omega^-
 =\lim_{j\rightarrow\infty} I^{1}_{j}+\lim_{j\rightarrow\infty} I^{2}_{j}
\end{align*}
Observe that $\supp \phi_{\xi,r_{j}}\subset T_{\xi,r_{j}}^{-1}(B(0,M))=B(\xi,Mr_{j})$, and since $\xi\in \Gamma\subset \Lambda_1$,
\begin{align}
\limsup_{j\rightarrow\infty} c_{j} \,\omega^-(B(\xi,Mr_{j})) &= h(\xi)\,\limsup_{j\rightarrow\infty} c_j\,\omega^+(B(\xi,Mr_{j}))\notag \\ 
&= h(\xi)\,\limsup_{j\rightarrow\infty} \omega^+_{j}(B(0,M))\leq h(\xi)\,\omega^+_{\infty}(\overline B(0,M))<\infty. \label{e:lmsup}
\end{align}
Thus, using the condition that $\xi\in E^{*}$,
\begin{align*}
\limsup_{j\rightarrow\infty} &  I_{j}^{2}
 \leq \|\phi\|_{\infty} \limsup_{j\rightarrow\infty} c_{j}{\omega^-(B(\xi,Mr_{j})\setminus E)}\\
& \leq   \|\phi\|_{\infty} \left(\limsup_{j\rightarrow\infty} \frac{\omega^-(B(\xi,Mr_{j})\setminus E)}{\omega^-(B(\xi,Mr_{j}))}\right) \,\Bigl(\limsup_{j\rightarrow\infty} c_{j}\, \omega^-(B(\xi,Mr_{j})) \Bigr)\stackrel{\eqn{lmsup}}{=}0.
\end{align*}
On the other hand,
\begin{align*}
\lim_{j\rightarrow\infty} I_{j}^{1}
& =\lim_{j\rightarrow\infty} c_{j}\int_{E}h \,\phi_{\xi,r_{j}}d\omega^+\\
& = h(\xi)\, \lim_{j\rightarrow\infty}  c_{j}\int_{\d\Omega^+} \phi_{\xi,r_{j}}d\omega^+
-h(\xi)\,\lim_{j\rightarrow\infty} c_{j}\int_{\d\Omega^+\backslash E} \phi_{\xi,r_{j}}d\omega^+\\
&\quad+\lim_{j\rightarrow\infty}  c_{j}\int_{E} \ps{ h-h(\xi)} \,\phi_{\xi,r_{j}}d\omega^+\\
& = h(\xi)\lim_{j\rightarrow\infty} \int \phi\, d\omega^+_{j}-\lim_{j\rightarrow\infty} I_{j}^{3}+\lim_{j\rightarrow\infty} I_{j}^{4}.
\end{align*}
Since the first term on right hand side equals $h(\xi)\int \phi \,d\omega^+_{\infty}$, all that remains to show is that $\lim_{j\rightarrow\infty} I_{j}^{3}=
\lim_{j\rightarrow\infty} I_{j}^{4}=0$. This follows easily using that $\xi\in \Gamma$:
\[
\lim_{j\rightarrow\infty} I_{j}^{3}
 \leq \|\phi\|_{\infty}\, \limsup_{j\rightarrow\infty}  c_{j}\omega^+(B(\xi,Mr_{j}))\;\frac{\omega^+(B(\xi,Mr_{j})\setminus E)}{\omega^+(B(\xi,Mr_{j}))} \stackrel{\eqn{lmsup}}{=}0
\]
and analogously,
\[
\lim_{j\rightarrow\infty} I_{j}^{4}
 \leq \|\phi\|_{\infty}\, \limsup_{j\rightarrow\infty}  c_{j}\omega^+(B(\xi,Mr_{j}))\;\avint_{B(\xi,Mr_{j})} |h-h(\xi)|\, d\omega^+\stackrel{\eqn{lmsup}}{=}0
\]
\end{proof}

\vv
{ Next  we prove an analogue of some of the tools in \cite{KT99} and \cite{KT06}. We show that blow ups of harmonic measure and Green function converge to quantities similar to the harmonic measure and Green function with pole at infinity introduced by Kenig and Toro. 
}
\vv

\begin{lemma}
Let $\Omega^+\subset \bR^{n+1}$ be a $\Delta$-regular domain and $\Omega^-={\rm ext}(\Omega^+)$,
so that $\partial\Omega^+=\partial\Omega^-$. 
Let $\omega^{\pm}$ be the harmonic measures for $\Omega^\pm$. Let $\xi\in \d\Omega^{+}$ and
$\omega_\infty^+\in \Tan(\omega^+,\xi)$, with
$c_{j}\geq 0$, and $r_{j}\rightarrow 0$ such that $\omega_{j}^+=c_{j}T_{\xi,r_{j}}[\omega^+]\rightarrow \omega_{\infty}^+$. Let $\Omega_{j}^\pm=T_{\xi,r_{j}}(\Omega^\pm)$. Then there is a subsequence and a closed set $\Sigma\subset \R^{n+1}$ such that 
\begin{enumerate}[(a)]
\item $\d\Omega_{j}^+\cap K\rightarrow \Sigma\cap K$ in the Hausdorff metric for any compact set $K$.
\item $\Sigma^{c}=\Omega_{\infty}^+\cup \Omega_{\infty}^-$ where $\Omega_{\infty}^+$ is a nonempty open set and $\Omega_{\infty}^-$ is also open but possibly empty. Further, they
satisfy that for any ball $B$ with $\cnj{B}\subset \Omega_{\infty}^\pm$, a neighborhood of $\cnj{B}$ is contained in $\Omega_{j}^\pm$ for all $j$ large enough.
\item $\supp \omega_{\infty}^+\subset \Sigma$. 
\item Let $u^+(x)=G_{\Omega^+}(x,x^+)$ on $\Omega^+$ and $u^+(x)=0$ on $(\Omega^+)^{c}$. Set
\[u_{j}^+(x)=c_{j}\,u^+(xr_{j}+\xi)\,r_{j}^{n-1}.\]
Then $u_{j}^+$ converges uniformly on compact subsets of $\bR^{n+1}$ to a nonzero function $u_{\infty}^+$ that is harmonic on $\Omega_{\infty}^+$ and satisfies
\begin{equation}
u_{\infty}^+(y)\lec  \omega^+_{\infty}({\overline B(x,2\delta_0^{-1}r)})\,r^{1-n}\quad\mbox{ for } x\in \Sigma, \;\; r>0,\;\; \mbox{ and }\;\;y\in B(x,r)\cap \Omega^+_{\infty}.
\label{e:u<wr}
 \end{equation}
{and for any smooth compactly supported function $\phi$,
 \begin{equation}\label{e:ibp}
 \int_{\d\Omega^{+}} \phi \,d\omega_{\infty}^{+}=\int_{\Omega^{+}} \Delta\phi \,u_{\infty}^{+}\,dx
 \end{equation}}
 \end{enumerate}
Suppose now that $\Omega^-$ is also connected and $\Delta$-regular. Define analogously $\omega_{j}^{-}$, $u^-$, $u_j^-$ and $u^-_\infty$ and suppose that $\omega_{j}^{-}$ converges weakly to $\omega_{\infty}^{-}=h(\xi)\omega_{\infty}^{+}$ for some number $h(\xi)\in (0,\infty)$ (which happens, for example, if $\xi\in \Gamma$ where $\Gamma$ is as in \Lemma{samew}).
Then $\Omega_{\infty}^-\neq\varnothing$ and for a suitable subsequence, (d) holds for $u_{j}^-$, $u_{\infty}^-$, and
 $\Omega^-_\infty$. Furthermore, if we set $u_{\infty}=h(\xi) u_{\infty}^+-u_{\infty}^-$, then:
 \begin{enumerate}[(a)]
 \setcounter{enumi}{4}
\item $u_{\infty}$ extends to a continuous harmonic function on $\bR^{n+1}$.
\item $\Sigma=\{u_{\infty}=0\}$, with $u_\infty>0$ on $\Omega_\infty^+$ and $u_\infty<0$ on $\Omega_\infty^-$. Further, $\Sigma$ is a real analytic variety of dimension $n$.
\item $d\omega_{\infty}^+=-\frac{\partial u_\infty}{\partial \nu}\,d\sigma_{\d\Omega_{\infty}^+}$, where $\sigma_{S}$ stands for the surface measure on a surface $S$ and $\frac{\partial}{\partial \nu}$ is the outward normal derivative.
\end{enumerate}
\label{l:blowup1}
\end{lemma}

\begin{proof}
First, we establish a few estimates. Note that if both $\Omega^+$ and $\Omega^-$ are connected, then for $j$ large enough,
\begin{align}\label{e:phiw}
\int \phi \,d\omega_{j}^\pm
= c_{j} \int \phi_{\xi,r_{j}} d\omega^\pm&
=c_{j} \int \Delta \phi_{\xi,r_{j}} u^\pm\,dx
=c_{j}\,r_{j}^{-2}\int \Delta \phi\ps{\frac{x-\xi}{r_{j}}} u^\pm(x)\,dx\\ &
=c_{j}\,r_{j}^{n-1} \int \Delta \phi(y)\,u^\pm(r_{j}y+\xi)\,dy
=\int \Delta \phi \,\,u_{j}^\pm\,dx,\nonumber
\end{align}
since the pole lies outside the $\supp\phi_{\xi,r_{j}}$ for sufficiently large $j$. Moreover, if $B$ is centered on $\d\Omega_{j}^\pm$, then for $x\in B\cap\Omega_j$ and $j$ large enough,
\begin{align}\label{e:uij<w}
u_{j}^\pm(x) &
=c_{j}\,r_{j}^{n-1}\, u(r_{j}x+\xi)\\
&
\stackrel{\eqn{w>g}}{\lec} c_{j}\,r_{j}^{n-1}(r_{j}r(B))^{1-n}\, \omega^\pm(2\delta_0^{-1}r_{j} B+\xi)
=r(B)^{1-n}\,\omega_{j}^\pm(2\delta_0^{-1}B).\nonumber
\end{align}

\vv
Next we prove the statements (a)-(g):
\vv

(a) This follows from a standard diagonalization argument, and so we omit its proof. \\

(b) First we show that there are balls $B^\pm$ so that, by passing to a subsequence, $B^{\pm}\subset \Omega_{j}^\pm$ for all $j$ large.

We will focus first on showing the existence of $B^+$. Suppose there is no such ball.  Let $\phi$ be any continuous compactly supported nonnegative function for which $\int \phi\, d\omega_{\infty}^+\neq 0$, and let $M>0$ be so that $\supp \phi \subset B(0,M)$. Thus, there must be $x_{0}\in B(0,M)\cap \supp \omega_{\infty}^+$. Let $\delta_{j}=\sup\{\dist(x,(\Omega_{j}^+)^{c}):x\in \supp \phi\}$, which goes to zero by assumption. For $x\in \supp\phi$, let $\zeta_{j}(x)\in (\Omega_{j}^+)^c$ be closest to $x$, so 
that $|x-\zeta_{j}(x)|\leq \delta_{j}$. 
Notice that for all $x\in\supp\phi$,  $|x-x_{0}|\leq |x|+|x_{0}|< 2M$ and also 
$|x-\zeta_j(x)|\leq \delta_j< 2M$ for $j$ big enough.
 Thus, for $j$ large enough, taking into account that $\zeta_j(x) \in\d\Omega_j^+$ if $x\in\Omega_j^+$, we get
\begin{align*}
0<\int \phi\, d\omega_{j}^+ 
& =\int_{\Omega_j^+} \Delta \phi \,u_{j}^+\,dx
\stackrel{\eqn{holder}}{\lec} \int_{\Omega_j^+} |\Delta \phi| \biggl(\,\sup_{B(\zeta_{j}(x),2M)} u_{j}^+\biggr) \ps{\frac{|x-\zeta_{j}(x)|}{2M}}^{\alpha}\,dx\\
& \leq \int |\Delta \phi|\,dx\, \biggl(\,\sup_{B(x_{0},4M)} u_{j}^+\biggr)\ps{\frac{\delta_{j}}{2M}}^{\alpha}\\
& \stackrel{\eqn{uij<w}}{\lec} \int |\Delta \phi| \,dx\;\omega_{j}^+(B(x_{0},8\delta_0^{-1}M)) \,(4M)^{1-n}\ps{\frac{\delta_{j}}{2M}}^{\alpha}
\end{align*}
and thus
\begin{align*}
0<\int \phi \,d\omega_\infty^+ 
& \lec \limsup_{j\rightarrow\infty}  \int |\Delta \phi|\,dx\; \omega_{j}^+(B(x_{0},8\delta_0^{-1}M)) (4M)^{1-n}\ps{\frac{\delta_{j}}{2M}}^\alpha\\
& \lec_{M,\phi} \ps{\limsup_{j\rightarrow\infty}  \omega_{j}^+(B(x_{0},8\delta_0^{-1}M))}  \lim_{j}\delta_{j}^{\alpha}\\
&\leq \omega^+_{\infty}(\overline B(x_{0},8\delta_0^{-1}M)) \cdot 0=0,
\end{align*}
which is a contradiction. Thus, there is $B^+\subset \Omega_{j}^+$ for all large $j$ (after passing to a subsequence). In case that $\Omega^-$ is a $\Delta$-regular domain, if $\xi\in \Gamma$, we run the same argument on $\Omega_{j}^-$, recalling from the previous lemma that $\omega^-_{j}\rightarrow h(\xi)\omega_{\infty}^+$.

Let $\mathscr{Q}$ be the collection of open balls with rational centers and rational radii whose closure is contained in $\Sigma^{c}$. By the previous claim, $\mathscr{Q}\neq\varnothing$. 
Let $B\in\mathscr{Q}$, so that for some $\alpha_{B}>1$, $\alpha_{B} B\subset \Omega^+_j\cup\Omega^-_j$ for all sufficiently large $j$. In particular, either $\alpha_{B} B\subset \Omega_{j}^+$ for infinitely many $j$, or $\alpha_{B} B\subset \Omega_{j}^-$ for infinitely many $j$. By a diagonalization argument, we can pass to a subsequence so that for all such balls $B$, $\alpha_{B}B\subset \Omega_{j}^+$ for all but finitely many $j$ or  $\alpha_{B}B\subset \Omega_{j}^-$ for all but finitely many $j$. Let $\mathscr{Q}^+$ be those balls in $\mathscr{Q}$ that are contained in all but finitely many $\Omega_{j}^+$ (after passing to this subsequence), $\mathscr{Q}^-=\mathscr{Q}\backslash \mathscr{Q}^+$, and set
\[
\Omega_\infty^\pm=\bigcup_{B\in \mathscr{Q}^\pm}B.
\]
By the previous claim, $\Omega_\infty^+\neq\varnothing$, and also $\Omega_\infty^-\neq\varnothing$ if $\Omega^-$ is a connected $\Delta$-regular domain. It is easy to check that $\Omega_\infty^+$ and
$\Omega_\infty^-$ satisfy the properties stated in (b).\\

(c) To prove this we consider a ball $B\subset \cnj{B} \subset \Sigma^{c}$. Then 
\[
\omega^+_{\infty}(B)\leq \liminf_{j\rightarrow\infty} \omega_{j}^+(B)\leq \liminf_{j\rightarrow\infty}\omega_{j}^+((\d\Omega_{j}^+)^{c})=0.\]
Thus, $\supp \omega^+\subset \Sigma$. \\

(d) 
Let $B\subset \Omega_{\infty}^+$ be a ball centered at $x_B$ such that $r(B)=\dist(x_{B},\d\Omega_{\infty}^+)/2$. For $j$ large enough, there is $y_{j}\in 3B\cap \d\Omega_{j}^+$. Then
\begin{align*} \limsup_{j\rightarrow\infty} \sup_{B} u_{j}^+ 
& \leq \limsup_{j\rightarrow\infty}  \sup_{B(y_{j},6r(B))} u_{j}^+
\stackrel{\eqn{uij<w}}{\lec} \limsup_{j\rightarrow\infty}  r(B)^{1-n} \omega_{j}^+(B(y_{j},12\delta_0^{-1}r(B)))\\
& \leq \limsup_{j\rightarrow\infty}  r(B)^{1-n} \omega_{j}^+(24\delta_0^{-1}B)
\leq r(B)^{1-n} \omega_\infty^+ \bigl(\cnj{24\delta_0^{-1}B}\bigr)<\infty.\end{align*}
Thus, $u_{j}^+$ is uniformly bounded on compact subsets of $\Omega_{\infty}^+$ and thus we may pass to a subsequence so that it converges uniformly on compact subsets of $\Omega_{\infty}^+$ to a function $u_{\infty}^+$ harmonic in $\Omega_\infty^+$. Define $u_{\infty}^+=0$ on $(\Omega_{\infty}^+)^{c}$. We now claim that $u_{j}^+\rightarrow u_{\infty}^+$ uniformly on compact subsets of $\bR^{n+1}$. 

To prove our claim let $M,\ve>0$ and consider the sets 
\[F=\{x\in B(0,M): \dist(x,(\Omega_{\infty}^+)^{c})\geq \delta \}\quad \mbox{ and }\quad G=B(0,M)\backslash F.\] 
For $x\in G \cap \Omega^+_\infty$, let $x'\in \d\Omega_{\infty}^+$ be closest to $x$, so that $|x-x'|<\delta$. There is $x_{j}\in \d\Omega_{j}^+$ converging to $x'$, and so, for $j$ big enough,
\begin{align*}
u_{j}^+(x) &
\stackrel{\eqn{holder}}{\lec} \sup_{B(x_{j},M/2)} u_{j}^+\ps{\frac{|x-x_{j}|}{M/2}}^{\alpha}\\
&
\stackrel{\eqn{uij<w}}{\lec}  \omega_{j}^+(B(x_{j},\delta_0^{-1}M))\, (M/2)^{1-n} \ps{\frac{\delta}{M}}^{\alpha}
\lec_{M} \omega_{j}^+(B(0,2M\delta_0^{-1})) \,\delta^{\alpha}.
\end{align*}
The same estimate holds trivially in the case $x\in G \backslash \Omega^+_\infty$. Thus, for every $x \in G$,
$$u_{\infty}^+(x)\lec_{M}\omega_\infty^+(\overline B(0,2\delta_0^{-1}M)) \delta^{\alpha},$$ 
and so
\[
\limsup_{j\rightarrow\infty} \sup_{G} |u_{j}^+-u_{\infty}^+|
\lec \omega_\infty^+(\overline B(0,2\delta_0^{-1}M)) \,\delta^{\alpha}.\]
On the other hand, since $F$ has compact closure in $\Omega_{\infty}^+$, 
\[
\limsup_{j\rightarrow\infty} \sup_{F} |u_{j}^+-u_{\infty}^+|=0.\]
Hence, for any $\delta>0$, since $B(0,M)=F\cup G$, the last two inequalities imply
\[
\limsup_{j\rightarrow\infty} \sup_{B(0,M)} |u_{j}^+-u_{\infty}^+|\lec \omega_\infty^+(\overline B(0,2\delta_0^{-1}M))\, \delta^{\alpha},\]
which implies $u_{j}^+\rightarrow u_{\infty}^+$ uniformly on $B(0,M)$. Since this holds for each $M>0$, the claim follows. In particular, $u_{\infty}^+$ is continuous on all of $\bR^{n+1}$. 

The estimate \rf{e:u<wr} follows by arguments analogous to the ones above. { Equation \eqn{ibp} now follows from uniform convergence and \eqn{phiw}.}
\\

(e) Let $u_{\infty}=u_{\infty}^+-h(\xi)^{-1} u_{\infty}^-$. To show that $u_{\infty}$ is harmonic, let $\phi\in C_{c}^{\infty}(\bR^{n+1})$. Then, since $\omega_{j}^{-}\rightarrow \omega_{\infty}^{-}=h(\xi)\omega_{\infty}^{+}$ by assumption, 

\begin{align*} \int \Delta \phi \,u_{\infty}\,dx
& =\lim_{j\rightarrow\infty} \int \Delta \phi \, (u_{j}^+-h(\xi)^{-1}u_{j}^-)\,dx\\
& =\lim_{j\rightarrow\infty} \left(\int \phi \,d\omega_{j}^+-h(\xi)^{-1}\int \phi \,d\omega_{j}^-\right)\\
& =\int\phi\, d\omega_{\infty}^+-h(\xi)^{-1}\int \phi \,d\omega_{\infty}^-
=\int\phi \,d \omega_{\infty}^+-\int\phi \,d \omega_{\infty}^+=0.\end{align*}
and so $u_{\infty}$ is a harmonic function on $\bR^{n+1}$. \\

(f) By construction it is clear that $u_\infty=0$ in $\Sigma$. To show that $u_\infty$ does not vanish out of $\Sigma$ first we
check that $u_\infty$ is not identically $0$. To see this, we take a non-negative and smooth compactly supported function $\phi$
such that $\int \phi\,d\omega_\infty^+>0$. By \rf{e:phiw} we have
$$\int \phi\,d\omega_j^+ = \int\Delta\phi\,u_j^+\,dx,$$
and so letting $j\to\infty$, we get
$$0<\int \phi\,d\omega_\infty^+ = \int\Delta\phi\,u_\infty^+\,dx.$$
This implies that $u_\infty^+$ is not identically zero, and thus neither is $u_\infty$.

By the definition of $u_\infty$, it is clear that $u_\infty\geq0$ on $\Omega_\infty^+$ and $u_\infty\leq0$ on $\Omega_\infty^-$.
Suppose there is $z\in \Omega_{\infty}^+$ such that $u_{\infty}^+(z)=0$, say. Then by the mean value property, $u_\infty^+$
should vanish in some ball $B\subset\Omega^+$ centered at $z$.
 But since $u_{\infty}$ coincides with $u_\infty^+$ on $B$, and $u_{\infty}$ is harmonic in
 the whole $\bR^{n+1}$, 
 this should vanish identically in $\R^{n+1}$, which is a contradiction. An analogous argument shows that
$u_\infty^->0$ on $\Omega^-$, and completes the proof of $\Sigma=\{u_\infty=0\}.$

On the other hand, since $u_\infty$ is harmonic, it is also real analytic, and thus $\Sigma$ is a real analytic variety.
Its dimension is less that $n+1$ because $\Sigma\neq\R^{n+1}$. 
 To show that it has dimension equal to
$n$, consider two balls $B_1\subset \Omega^+$ and $B_2\subset\Omega^-$, so that $u_\infty>0$ on $B_1$ and $u_\infty<0$ on $B_2$.
By continuity, each segment $L$ joining $B_1$ and $B_2$ should contain a point where $u_\infty$ vanishes. That is, 
$L\cap\Sigma\neq\varnothing$. This shows that $\HH^n(\Sigma)>0$, and hence $\Sigma$ has dimension at least $n$.
\\


(g) This follows from Theorem \ref{thm:Gauss-Green} once we show that $\Omega_{\infty}^+$ is a set of locally finite perimeter and $\cH^n ({\d\Omega_{\infty}^+}\setminus {\d^{*}\Omega_{\infty}^+})=0$, where $\d^{*}\Omega_{\infty}^+\subset \d\Omega_{\infty}^+$ stands for the reduced boundary of $\Omega_{\infty}^+$. Note that $\partial\Omega_\infty^+$ is real analytic and by Theorem 3.4.8 in \cite{Federer} it has locally finite $\cH^n$ measure. Therefore, Theorem \ref{thm:BV-Criterion} implies that $\Omega_\infty^+$ has locally finite perimeter.

We claim that $\cH^n ({\d\Omega_{\infty}^+}\setminus {\d^{*}\Omega_{\infty}^+})=0$. By Lojasiewicz's structure theorem for real analytic varieties (see e.g. \cite[Theorem 6.3.3,  p. 168]{KrPa}), if $Q$ is a small enough neighborhood of a point $x_0 \in \d\Omega_{\infty}^+$, we have that 
$$Q \cap \d\Omega_{\infty}^+= V^n \cup V^{n-1} \cup \dots \cup V^0,$$
where $V^0$ is either the empty set or the singleton $\{x_0\}$ and for each $k \in \{1, \dots, n\}$, we may write $V^k$ as a finite, disjoint union $V^k = \bigcup_{j=1}^{N_k} \Gamma^k_j$, of $k$-dimensional real analytic submanifolds. Further, for each $1 \leq k \leq n$, 
$$Q \cap \overline {V^k} \supset V^{k-1} \cup \dots \cup V^0,$$
 which, in fact, says that the lower dimensional varieties cannot occur as isolated sets  (stratification). Moreover,  for $1 \leq k \leq n$ and $1 \leq j \leq N_k$, we have that $Q \cap \d \Gamma^k_j$ is a union of sets of the form $\Gamma^\ell_i$, for $1 \leq \ell <k$ and $1 \leq i \leq N_\ell$ and possibly $V^0$. Notice now that, by the mean value property, the $n$-dimensional varieties should separate the connected components of $\{u_\infty>0\}$ and $\{u_\infty<0\}$. Therefore, since the lower dimensional varieties have $\cH^n$-measure zero, it is clear that $\d\Omega_\infty^+=\d^*\Omega_\infty^+ \cup N$, where $\cH^n(N)=0$, which proves our claim.

In light of Theorem \ref{thm:Gauss-Green}, for  $\cH^n$-a.e. $x \in \d^{*}\Omega_{\infty}^+$ there exists a unique measure theoretic unit outer normal $\nu_{\Omega_{\infty}^+}(x)$ such that 
\[
\int_{\d\Omega_{\infty}^+} \phi \, d\omega_{\infty}^+=\int_{\Omega_{\infty}^+} \Delta \phi\, u_{\infty}^+\,dx=-c_n\int_{\d^*\Omega_{\infty}^+} \phi \, (\nu_{\Omega_{\infty}^+} \!\cdot\! \nabla u^+_{\infty})\, d\cH^n,\]
for all $\phi\in C_{c}^{\infty}(\bR^{n+1})$, with $c_n\,d\HH^n|_{\d_*\Omega_{\infty}^+} = d\sigma_{\d\Omega_{\infty}^+}$. The statement (g) follows from this fact and the identity above.
\end{proof}

\vv
A corollary of the previous lemma is the following.

\begin{lemma}
Let $\Omega^+$ and $\Omega^-$ be as in Lemma \ref{l:blowup1}.
Let $\xi\in \Gamma$. For every $\omega\in \Tan(\omega^+,\xi)$, there is a harmonic function $u$ on $\bR^{n+1}$ such that 
\begin{equation}\label{eq5.5}
d\omega = { - \nu_{\Omega} \!\cdot\! \nabla u}\; d\HH^n|_{\Sigma},
\end{equation}
\begin{equation}\label{eq5.6}
\supp \omega\subset \Sigma=\{u=0\}=\d\Omega, \;\; \Omega=\{u>0\}
\end{equation}
\begin{equation}
u(y)\lec  \omega({\overline B(x,2\delta_0^{-1}r)})\,r^{1-n}\quad\mbox{ for } x\in \Sigma, \;\; r>0,\;\; \mbox{and }\;y\in B(x,r)\cap \Omega,
\label{e:u2<wr}
 \end{equation}
 and
 \begin{equation}
|u(y)|\lec  h(\xi)\,\omega({\overline B(x,2\delta_0^{-1}r)})\,r^{1-n}\quad \mbox{ for } x\in \Sigma, \;\; r>0,\;\; \mbox{and }\;y\in B(x,r)\cap {\rm ext}(\Omega). \label{e:u3<wr}
\end{equation}
Moreover, there is a subsequence of $\{r_{j}\}$ so that $T_{\xi,r_{j}}(\d\Omega^+)\rightarrow \{u=0\}$ locally in the Hausdorff metric.
\label{l:blowup}
\end{lemma}

\vv

\begin{lemma}\label{l:tanconnect}
Let $\Omega^+$ and $\Omega^-$ be as in Lemma \ref{l:blowup1} and let $\xi\in\Gamma$.
Let $\cF$ be given by \rf{eqflat}.
If $\Tan(\omega^+,\xi)\cap \cF\neq\varnothing$, then $$\lim_{r\rightarrow 0} d_{1}(T_{\xi,r}[\omega^+],\cF)= 0.$$ In particular, $\Tan(\omega^+,\xi)\subset \cF$. 
\end{lemma}

 The proof combines ideas from Theorem 2.15 and Lemma 4.1 in \cite{KPT09}. In this work the proof relies
on the compactness of the cone of tangent measures. In our situation we cannot assume compactness and we 
overcome this difficulty by working specifically with the
 flat measures $\cF$ {\it and} by using the additional information on the tangent measures described by the previous lemma.

\begin{proof}
Let $c_{j}>0$ and $r_{j}\downarrow 0$ be such that $c_{j}T_{\xi,r_{j}}[\omega^+]\rightarrow \mu\in \cF$. Then, given an arbitrary $\ve>0$,
\begin{equation}\label{e:distrj}
d_1(T_{\xi,r_{j}}[\omega^+],\cF)=d_1(c_{j}T_{\xi,r_{j}}[\omega^+],\cF)<\ve
\end{equation}
if $j$ is big enough.
Assume for the sake of a contradiction that there is $s_{j}\downarrow 0$ so that
\begin{equation}\label{e:distsj}
d_1(T_{\xi,s_{j}}[\omega^+],\cF)>\ve
\end{equation}
We can assume $s_{j}<r_{j}$. Then by \eqn{distrj} and \eqn{distsj}, let $\rho_{j}\in (s_{j},r_{j})$ be the maximal number such that
\begin{equation}\label{e:distrhoj}
d_1(T_{\xi,\rho_{j}}[\omega^+],\cF)=\ve.
\end{equation}
Then by the maximality of $\rho_{j}$,
\begin{equation}\label{e:t>rho}
\sup_{t\in [\rho_{j},r_{j}]} d_1(T_{\xi,t}[\omega^+],\cF)\leq \ve.
\end{equation}
We claim $\rho_{j}/r_{j}\rightarrow 0$. If not, then we may pass to a subsequence so that $\rho_{j}/r_{j}\rightarrow t\in (0,1)$, and so 
\[
c_{j} T_{\xi,\rho_{j}}[\omega^+]=T_{0,\rho_{j}/r_{j}}\bk{c_{j}T_{\xi,r_{j}}[\omega^+]}\rightarrow T_{0,t}[\mu]\in \cF\]
which contradicts \eqn{distrhoj}. Thus, $\rho_{j}/r_{j}\rightarrow 0$, and so \eqn{t>rho} implies that for $1\leq  \alpha<r_{j}/\rho_{j}$ and $j$ large, if we set $\omega_{j}=T_{\xi,\rho_{j}}[\omega^+]$, then
\begin{equation}\label{e:ojF}
d_{\alpha}(\omega_{j},\cF)=d_{\alpha}(T_{\xi,\rho_{j}}[\omega^+],\cF)
\stackrel{\eqn{dr}}{=}d_{1}(T_{\xi,\alpha \rho_{j}}[\omega^+],\cF)\leq \ve.
\end{equation}

Let $r\geq 1$ be such that $2r<r_{j}/\rho_{j}$. Let $\mu_{j}\in \cF$ be such that $F_{2 r}(\mu_{j})=1$ and
\begin{equation}
F_{r}\ps{\frac{\omega_{j}}{F_{2 r}(\omega_{j})},\mu_{j}}
\leq F_{2 r}\ps{\frac{\omega_{j}}{F_{2 r}(\omega_{j})},\mu_{j}}
\stackrel{\eqn{ojF}}{<} 2\ve.
\end{equation}
Thus,
\begin{equation}\label{e:F-e}
F_{r}(\mu_{j})-2\ve
\leq \frac{F_{r}(\omega_{j})}{F_{2 r}(\omega_{j})}\leq  F_{r}(\mu_{j})+2\ve.
\end{equation}
Since $\mu_{j}= b_{j}\cH^{n}|_{V_{j}}$ for some $b_{j}>0$ and an $n$-plane $V_{j}$, for any $s>0$,
\[
F_s(\mu_{j})=b_{j} \frac{c_{n}}{n} s^{n+1}\]
and so $F_{r}(\mu_{j})=
2^{-(n+1)}F_{2r}(\mu_{j})$ and thus 
\[
2^{-(n+1)}\,F_{2r}(\mu_{j})-2\ve
\leq \frac{F_{r}(\omega_{j})}{F_{2 r}(\omega_{j})}\leq  2^{-(n+1)}\,F_{2r}(\mu_{j})+2\ve.
\]
Recalling that $ F_{2 r}(\mu_{j})= 1$, we deduce
$$2^{-(n+1)}-2\ve
\leq \frac{F_{r}(\omega_{j})}{F_{2 r}(\omega_{j})}\leq  2^{-(n+1)}+2\ve.
$$
We choose now $\ve = \frac18\,2^{-(n+1)} = 2^{-(n+4)}$. With this particular choice we get
$$\frac34 \,2^{-(n+1)} \leq 2^{-(n+1)}\Bigl(1- \frac14\Bigr)
\leq \frac{F_{r}(\omega_{j})}{F_{2 r}(\omega_{j})}\leq  2^{-(n+1)} \Bigl(1+\frac14\Bigr) \leq 
\frac43\,2^{-(n+1)}.
$$
Let $\beta=\log_2\frac43$, so that $\beta\in(0,1)$ and 
$$2^{-(n+1+\beta)}
\leq \frac{F_{r}(\omega_{j})}{F_{2 r}(\omega_{j})}\leq  2^{-(n+1-\beta)}.
$$

Iterating, we see that for any $\ell\in \bN$ for which $2^{\ell}<r_{j}/\rho_{j}$,
\[
2^{-(n+1+\beta)\ell}
\leq \frac{F_{1}(\omega_{j})}{F_{2^{\ell}}(\omega_{j})}\leq  2^{-(n+1-\beta)\ell}.\]
Since $r_{j}/\rho_{j}\rightarrow\infty$, we have that 
\begin{equation}\label{e:f1/fj}
2^{-(n+1+\beta)\ell}
\leq \liminf_{j} \frac{F_{1}(\omega_{j})}{F_{2^{\ell}}(\omega_{j})} 
\leq \limsup_{j\rightarrow\infty} \frac{F_{1}(\omega_{j})}{F_{2^{\ell}}(\omega_{j})}
  \leq  2^{-(n+1-\beta)\ell}.\end{equation}
Thus, if we set $\nu_{j}=\omega_{j}/F_{1}(\omega_{j})$ and let $\rho>0$, and pick $\ell$ so that $2^{\ell}>2\rho$, then for $j$ large,
\begin{align*}
\limsup_{j\rightarrow\infty} \nu_{j}(B(0,\rho)) &
\leq \limsup_{j\rightarrow\infty} \nu_{j}(B(0,2^{\ell}/2))\\
& \leq \limsup_{j\rightarrow\infty} 2F_{2^{\ell}}(\nu_{j})=2\frac{F_{2^{\ell}}(\omega_{j})}{F_{1}(\omega_{j})}
\stackrel{\eqn{f1/fj}}{\leq} 2^{(n+1+\beta)\ell + 1}.
\end{align*}
Therefore, $\nu_{j}$ has a subsequence that converges weakly to some measure $\omega\in \Tan(\omega^+,\xi)$. Further, $\omega$ satisfies
\[
2^{(n+1-\beta)\ell}
\leq F_{2^{\ell}}(\omega)\leq  2^{(n+1+\beta)\ell} \;\; \mbox{ and }\;\; F_{1}(\omega)=1.\]

Let $u$ be the harmonic function on $\bR^{n+1}$ satisfying the conclusions of \Lemma{blowup}. For a multiindex $\alpha$ with $|\alpha|=m\geq 2$, we have by the Cauchy estimates that, for $\ell\in \bN$, 
\begin{align*}
|\d_{\alpha} u(0)|(2^{\ell} )^{m}
& \lec \sup_{B(0,2^{\ell})} |u|
\lec \omega(B(0,2\delta_0^{-1}2^{\ell}))(2^{\ell} )^{1-n}
\lec F_{\delta_0^{-1}2^{\ell+2}}(\omega) (2^{\ell} )^{-n}\\
&\lesssim 2^{(n+1+\beta)\ell} (2^{\ell})^{-n}
=2^{\ell+\beta \ell}
\end{align*}
Since $\beta<1$, letting $\ell\rightarrow\infty$, we get
\[
|\d_{\alpha} u(0)|\lec \liminf_{\ell\rightarrow \infty} 2^{(\beta+1-m)\ell}=0.\]
Thus, the second order Taylor coefficients and higher are all zero. Hence, since $u$ is real analytic, $u$ is linear, and in particular, $\omega\in \cF$, by \rf{eq5.5} and \rf{eq5.6}. Therefore, 
\[\ve \stackrel{\eqn{distrhoj}}{=} d_{1}(T_{\xi,\rho_{j}}[\omega^+],\cF)=d_{1}(\omega_{j},\cF)\rightarrow 0,\]
which gives a contradiction.
\end{proof}

\vvv
We now finish the proof of \Theorem{alpha0}. Let 
\[
G_{1}=\bigl\{\xi\in \Gamma: \mbox{ for every } \nu\in \Tan(\omega^+,\xi) \;\mbox{ with }\;  \zeta\in \supp \nu, \;\; \Tan(\nu,\zeta)\subset\Tan(\omega^+,\xi)\bigr\}\]
and 
\[G_{2}=\bigl\{\xi\in \Gamma: \Tan(\omega^+,\xi)\neq\varnothing\bigr\}.\]
Then $\omega^+(\Gamma\backslash (G_{1}\cap G_{2}))=0$ by \Theorem{tme} and \Theorem{ttt}. Let $\xi\in G_{1}\cap G_{2}$. Since $\xi\in G_{2}$, there is $\omega_{\infty}\in \Tan(\omega^+,\xi)$ and its support is an $n$-dimensional analytic variety. Hence, there is an open set in $\supp\omega_{\infty}$ on which $\omega_{\infty}=g\,\cH^{n}|_{M}$ where $M$ is a smooth $n$-dimensional surface. In particular, for any $x\in M$, $\Tan(\omega_{\infty},x)\subset \cF$. Since $\xi\in G_{1}$, this implies $\Tan(\omega^+,\xi)\cap \cF\neq\varnothing$. Thus, by \Lemma{tanconnect}
$\Tan(\omega^+,\xi)\subset \cF$ and
$\Sigma$ is an $n$-plane.

Suppose that there is a sequence $r_{j}\to0$ so that $\beta_{\d\Omega^+,\infty}(\xi,r_{j})\geq \ve>0$ for some $\ve>0$. 
By \Lemma{blowup1}, there is a subsequence
such that $T_{\xi,r_{j}}(\d\Omega^+)$ converges in the Hausdorff metric to $\Sigma$, this implies $\beta_{\d\Omega^+,\infty}(\xi,r_{j})\rightarrow 0$, and we get a contradiction. Thus $\beta_{\d\Omega^+,\infty}(\xi,r)\rightarrow 0$ as $r\rightarrow 0$ for each $\xi\in G_1\cap G_2$, with
$G_1\cap G_2\subset \Gamma\subset E$ having full harmonic measure in $E$.

We claim now that if $\Tan(\omega^+,\xi)\subset \cF$, then
\begin{equation*}
\lim_{r\rightarrow 0}\wt\beta_{\overline{\Omega^+},\infty}(\xi,r)=0.
\end{equation*}
If this fails, then there is an $\ve>0$ and a 
sequence $r_j\to0$ such that
\begin{equation}\label{eqclaim*594}
\inf_S\,\frac{\dist_H\bigl(\overline{\Omega^+}\cap B(\xi,r_j),S\cap B(\xi,r_{j}))}{r_j}\geq \ve,
\end{equation}
where the infimum is taken over all halfspaces $S$ whose boundaries contain $\xi$. We consider 
now a subsequence $r_{j_k}$ such that the measures $\omega_{j_k}^+$ converge weakly to some measure
$\omega\in\cF$.
The arguments for the proof of Lemma \ref{l:tanconnect} show that if 
$\omega_{j_k}^+=T_{\xi,r_{j_k}}[\omega^+]$ and $\omega_{j_k}^+$ converges weakly to some measure
$\omega^+_\infty\in \cF$, then the associated function $u_\infty$ from Lemma 
\ref{l:blowup1} must be linear. Then the statement (f) from the same lemma asserts that 
$\Omega^+_\infty$ and $\Omega^-_\infty$ are disjoint half-spaces with boundary $\Sigma=\{u_\infty=0\}$. Taking into account that $\overline{\Omega^+_{j_k}} = \{u_{j_k} \geq0\}$, where $u_{j_k}=u_{j_k}^+-u_{j_k}^-$, and that
$u_{j_k}$ converges uniformly on 
compact subsets to $u_\infty$, it easy to check that
$$\frac{\dist_H\bigl(\overline{\Omega^+}\cap B(\xi,r_{j_k}),S\cap B(\xi,r_{j_{k}})\bigr)}{r_{j_k}} \to 0
\quad \mbox{ for $S=\xi + \overline{\Omega^+_\infty},$}$$
which contradicts \rf{eqclaim*594} (because $ \overline{\Omega^+_\infty}$ is a half-space
whose boundary contains $\xi$). This proves our claim and concludes the proof of \Theorem{alpha0}.

\vv

\section{The proof of \Theorem{main}}

Under the assumptions of the theorem, we will prove first that if
 $E\subset \d\Omega^+$ and $\omega^{+}\ll\omega^{-}\ll\omega^{+}$ on $E$, then $E$ contains an $n$-rectifiable subset $F$ on which $\omega^\pm$ are mutually absolutely continuous with respect to $\cH^{n}$. 
So for the moment, unless otherwise stated, we assume that $\Omega^+$ and $\Omega^-$ are as in
  \Theorem{main} and that $E\subset \d\Omega^+$ is a Borel set such that $\omega^{+}\ll\omega^{-}\ll\omega^{+}$ on $E$.
 \vv
 
 Given $\gamma>0$, a Borel measure $\mu$ and a ball $B\subset\R^d$, we denote
$$P_{\gamma,\mu}(B) = \sum_{j\geq0} 2^{-j\gamma}\,\Theta_\mu(2^jB),$$
where $\Theta_\mu(B) = \frac{\mu(B)}{r(B)^n}$, so that $P_{1,\mu}(B) = P_{\mu}(B)$.
Note that $P_{\gamma,\mu}(B) \leq P_{\Gamma,\mu}(B)$ if $\gamma>\Gamma$. It is immediate to check
that if $\|\mu\|<\infty$, then $P_{\gamma,\mu}(B)<\infty$ for any ball $B$. Indeed, we just take into
account that
\begin{equation}
P_{\gamma,\mu}(B)
=\sum_{j\geq0} 2^{-j\gamma}\,\Theta_\mu(2^jB)\leq 
\sum_{j\geq0} 2^{-j\gamma}\,\frac{\|\mu\|}{(2^j\,r(B))^n}<\infty.
\end{equation}

\vv

Given $a,\gamma>0$, we say that a ball $B$ is {\it $a$-$P_{\gamma,\mu}$-doubling} if
$$P_{\gamma,\mu}(B) \leq a \,\Theta_\mu(B).$$

\begin{lemma}\label{l:adoubling}
There is $\gamma_{0}\in (0,1)$ so that the following holds. Let $\Omega\subset \bR^{n+1}$ be any domain and $\omega$ its harmonic measure. For all $\gamma>\gamma_{0}$, there exists some big enough constant $a=a(\gamma,n)>0$ such that  for $\omega$-a.e.
$x\in\R^{n+1}$ there exists a sequence of $a$-$P_{\gamma,\omega}$-doubling balls $B(x,r_i)$, with
$r_i\to0$ as $i\to\infty$.
\end{lemma}

\begin{proof}
For $m\geq1$, let 
\begin{equation}
Z_m:= \{x\in \d\Omega: \mbox{ for all }j\geq m, \;\; B(x,2^{-j})\mbox{ is not $a$-$P_{\gamma,\omega}$-doubling}\}.
\end{equation} 
So it is enough to show that $\omega(Z_m)=0$ for all $m\geq1$.

Fix $m\geq 1$ and take $x\in Z_m$, so that 
\begin{equation}\label{e:t<ap}
\Theta_{\omega}(B(x,2^{-j}))\leq a^{-1} \,P_{\gamma,\omega}(B(x,2^{-j})) \quad \mbox{ for all $j\geq m$.}
\end{equation}
Let $\alpha\in (0,1)$ to be chosen below. For $j\geq m$, 
\begin{align*}
& P_{\alpha\gamma,\omega}(B(x,2^{-j})) = \sum_{k\leq j} 2^{-\alpha\gamma(j-k)}
\Theta_{\omega}(B(x,2^{-k})) \\
& \leq a^{-1} \sum_{k:m\leq k\leq j} 2^{-\alpha\gamma(j-k)}
P_{\gamma,\omega}(B(x,2^{-k}))
+ \sum_{k\leq m} 2^{-\alpha\gamma(j-k)}
\Theta_{\omega}(B(x,2^{-k}))
\\
& = a^{-1} \sum_{k:m\leq k\leq j} 2^{-\alpha\gamma(j-k)} \sum_{h\leq k} 2^{-\gamma(k-h)}\Theta_{\omega}(B(x,2^{-h}))
+ 2^{-\alpha\gamma(j-m)} P_{\alpha\gamma,\omega}(B(x,2^{-m}))
\\
& \stackrel{\eqn{t<ap}}{\leq }
a^{-1} \sum_{h\leq j} \Theta_{\omega}(B(x,2^{-h})) \sum_{k:h\leq k\leq j}
  2^{-\gamma(k-h)-\alpha\gamma(j-k)} + 2^{-\alpha\gamma(j-m)} P_{\alpha\gamma,\omega}(B(x,2^{-m})).
\end{align*}
Observe now that
\begin{align*}
\sum_{k:h\leq k\leq j}
  2^{-\gamma(k-h)-\alpha\gamma(j-k)} & = 2^{\gamma h -\alpha\gamma j} \sum_{k:h\leq k\leq j}
  2^{-\gamma(1-\alpha) k} \\
  & \leq C(\gamma,\alpha) \,2^{\gamma h - \alpha \gamma  j}\, 2^{-(1-\alpha)\gamma h} = C(\gamma,\alpha)\, 2^{-\alpha\gamma(j-h)}.
\end{align*}
Thus we obtain
$$P_{\alpha\gamma,\omega}(B(x,2^{-j})) \leq C(\gamma,\alpha)\,a^{-1}P_{\alpha\gamma,\omega}(B(x,2^{-j}))
+ 2^{-\alpha\gamma(j-m)} P_{\alpha\gamma,\omega}(B(x,2^{-m})).$$
Hence, choosing $a\geq 2C(\gamma,\alpha)$ and recalling that $P_{\alpha\gamma,\omega}(B(x,2^{-j}))<\infty$, we  infer that
$$\Theta_{\omega}(B(x,2^{-j}))\leq P_{\alpha\gamma,\omega}(B(x,2^{-j})) \leq 2^{1-\alpha\gamma(j-m)} P_{\alpha\gamma,\omega}(B(x,2^{-m})).$$

Observe now that for all $x\in Z_m$,
$$P_{\alpha\gamma,\omega}(B(x,2^{-m}))\leq 
\sum_{k\geq0} 2^{-k\gamma}\,\frac{\|\omega\|}{(2^k\,2^{-m})^n}\leq C(m).$$
Then we get
$$\Theta_{\omega}(B(x,2^{-j}))\leq C(m)2^{-\alpha\gamma j}\quad \mbox{ for all $j\geq m$,}$$
which implies that
$$\omega(B(x,r))\leq C(m) r^{n+\alpha\gamma}\quad \mbox{ for all $x\in Z_m$ and all $r\leq 2^{-m}$.}$$
Thus, $\omega(A)\leq C(m)\,\HH_{\infty}^{n+\alpha\gamma}(A)$ for any $A\subset Z_{m}$. 

Recall that, for a measure $\mu$, 
\begin{multline*}
\dim \mu=\inf\{s: \mbox{ there is }F\subset \d\Omega \mbox{ so that }\HH^{s}(F)=0\mbox{ and } \\
\mu(F\cap K)=\mu(\d\Omega\cap K) \mbox{ for all compact sets }K\subset \R^{n+1}.
\end{multline*}
Let $s=n+\alpha\gamma$ and $F\subset \d\Omega$ be such that $\HH^{s}(F)=0$. Let $K$ be any compact subset of $Z_{m}$ with $\omega(K)>0$. Then $\omega(F\cap K)\leq \HH^{s}_{\infty}(F\cap K)=0$. Thus, $\dim \omega\geq s$. 

A well known theorem of Bourgain's asserts that there is $\ve(n)>0$ (not depending on $\Omega$) so that $\dim\omega<n+1-\ve(n)$ \cite{Bou87}. In particular, $s=n+\alpha \gamma<n+1-\ve(n)$, which is a contradiction if $\alpha\gamma\geq 1-\ve(n)$. So it just remains to notice that if $\gamma>1-\ve(n)$, we can now pick $\alpha\in (0,1)$ so that still $\alpha\gamma > 1-\ve(n)$. 
\end{proof}

\vv

From now on we assume that $a$ and $\gamma$ are fixed constants such that for
$\omega$-a.e.
$x\in\R^{n+1}$ there exists a sequence of $a$-$P_{\gamma,\omega^+}$-doubling balls $B(x,r_i)$, with
$r_i\to0$ as $i\to\infty$.

\vv
Recall that the harmonic measures $\omega^+$ and $\omega^-$ are mutually absolutely continuous on $E\subset\d\Omega^+ = \d\Omega^-$, and that $h$ denotes the density function $h(\xi)=\frac{d\omega^{-}}{d\omega^{+}}(\xi)$.\vv
 
For technical reasons we need now to introduce sets $E_m\subset E$ where $\wt\beta_{\overline{\Omega^{+}},\infty}(x,r)$
is uniformly small.
Given $m\geq1$, we denote by $E_m$ the subset of those $x\in E$ such that
$\wt\beta_{\overline{\Omega^{+}},\infty}(x,r)\leq 1/{100}$ for $0<r\leq 1/m$. By Theorem \ref{t:alpha0}, it turns
out that
\begin{equation}\label{eq:EMM}
\omega^+\biggl(E\setminus \bigcup_{m\geq1} E_m\biggr)=0.
\end{equation}

\vv

\begin{lemma}\label{lem:ver1}
Let $m\geq1$ and $\delta>0$. For $\omega^+$-a.e.\ $x\in E_m$, there is $r_{x}>0$ so that 
for any $a$-$P_{\gamma,\omega^+}$-doubling ball $B(x,r)$ with radius $r\leq r_x$ there exists a subset $G_m(x,r)\subset E_m\cap B(x,r)$ such that 
\begin{equation}\label{e:mainl1}
\Theta_{\omega^+}(B(z,t)) \lec \Theta_{\omega^+}(B(x,r)) \quad\mbox{for all $z\in G_m(x,r)$, $0<t\leq 2r$,}
\end{equation}
and so that $\omega^+ (B(x,r) \setminus G_m(x,r)) \leq \delta \,\omega^+(B(x,r))$.
\end{lemma}


\begin{proof} 
For $0<\delta<1$ and $k\in \N$, let $A_{\delta,k}$ be the set of $z\in E_m$ such that for $0<r<1/k$ we have 
\begin{equation}\label{eq6.5}
\avint_{B(z,r)}|h(y)-h(z)|\,d\omega^{+}(y)<\frac\delta4\, h(z).
\end{equation}
Since $h(z)>0$ for $\omega^+$-a.e.\ $z\in E_m$, by the Lebesgue differentiation theorem (see \cite[Corollary 2.14 (2) and Remark 2.15 (3)]{Mattila})
$$E_m=\bigcup_{k\geq1} A_{\delta,k}\cup Z,$$
with $\omega^+(Z)=0$. 
Then, for all $z\in A_{\delta,k}$ and $t<1/k$, we have
\begin{equation}
\left| \frac{\omega^{-}(B(z,t))}{\omega^{+}(B(z,t))}-h(z)\right|
=\left| \avint_{B(z,t)}(h(y)-h(z))\,d\omega^{+}(y)\right|<\frac\delta4\, h(z)<\frac14\, h(z).
\end{equation}
and so
\begin{equation}\label{eq6.99}
\frac{3}{4} \,h(z) \leq \frac{\omega^{-}(B(z,t))}{\omega^{+}(B(z,t))}\leq \frac{5}{4}\,h(z) \quad\mbox{ for }z\in A_{\delta,k},\, 0<t<1/k.
\end{equation}

Let $x\in A_{\delta,k}$ be a point of $\omega^{+}$-density for $A_{\delta,k}$ and let $r_{x}<1/k$ be such that 
\begin{equation}\label{eq6.8}
\omega^{+}(A_{\delta,k}\cap B(x,r))\geq \biggl(1-\frac\delta2\biggr)\,\omega^{+}(B(x,r))\quad \mbox{ for }0<r\leq r_{x}.
\end{equation}
Now set
\[G_m(x,r)=\{z\in B(x,r)\cap A_{\delta,k}: |h(z)-h(x)|\leq h(x)/2\}.\]
Then by Chebychev's inequality and \rf{eq6.5},
\[
\omega^{+}(B(x,r)\cap A_{\delta,k}\backslash G_m(x,r))
\leq \frac{2}{h(x)} \int_{B(x,r)} |h(z)-h(x)|d\omega^{+}(z)
\leq \frac\delta2\, \omega^{+}(B(x,r)),\]
and thus, together with \rf{eq6.8}, for $r\leq r_{x}$,
\begin{align*}
\omega^{+}(B(x,r)\setminus G_m(x,r)) & \leq \omega^{+}(B(x,r)\cap A_{\delta,k}\setminus G_m(x,r))
+ \omega^{+}(B(x,r)\setminus A_{\delta,k}) \\
& \leq \biggl(\frac\delta2 + \frac\delta2\biggr)\, \omega^{+}(B(x,r))
= \delta \,\omega^{+}(B(x,r)).
\end{align*}

We intend to show now that \rf{e:mainl1} holds for all $z\in G_m(x,r)$. 
Observe first that, for $z\in G_m(x,r)$ and $r\leq r_{x}$, 
\[
\frac12\,h(x)\leq h(z)\leq \frac32\,h(x)\]
and then, by \rf{eq6.99},
\begin{equation}\label{e:frach}
\frac{3}{8} h(x)\leq \frac{\omega^{-}(B(z,t))}{\omega^{+}(B(z,t))}\leq \frac{15}{8} h(x)\quad\mbox{ for all $z\in G_m(x,r)$ with $r,t\leq r_x$.}
\end{equation}
Recall that by Lemmas  \ref{l:beurling} and \ref{l:otherside}, for $0<t<2r$,
\begin{align*}
\frac{\omega^{+}(B(z,t))}{t^{n}}\,\frac{\omega^{-}(B(z,t))}{t^{n}}& \lec \gamma(z,2t)^{\frac{1}{2}}\\
&\leq 
\gamma(z,4r)^{\frac{1}{2}}
\lec \frac{\omega^{+}(B(z,16\delta_0^{-1}r))}{r^{n}}\,\frac{\omega^{-}(B(z,16\delta_0^{-1}r))}{r^{n}}.
\end{align*}
Take $0<r\leq \frac1{100}\delta_0r_x$ and $0<t\leq 2r$.
Applying \rf{e:frach} twice, we derive
$$\left(\frac{\omega^{+}(B(z,t))}{t^{n}}\right)^2 
\lesssim
\left(\frac{\omega^{+}(B(z,16\delta_0^{-1}r))}{r^{n}}\right)^2.$$
Since $z\in B(x,r)$, we have $B(z,16\delta_0^{-1}r)\subset B(x,32\delta_0^{-1}r)$, and then taking into account that
$B(x,r)$ is $a$-$P_{\gamma,\omega^+}$-doubling,
$$\omega^+(B(z,16\delta_0^{-1}r))\leq \omega^+(B(x,32\delta_0^{-1}r))\lesssim \omega^+(B(x,r)).$$
Therefore,
$$\frac{\omega^{+}(B(z,t))}{t^{n}}\lesssim
\frac{\omega^{+}(B(z,16\delta_0^{-1}r))}{r^{n}} \lesssim \frac{\omega^+(B(x,r))}{r^n},$$
which shows that \rf{e:mainl1} holds for all $z\in G_m(x,r)$, $t\leq 2r$, with $r$ such that $0<r\leq \frac1{100}\delta_0r_x$.
\end{proof}

\vv

Given $m\geq1$ and $\delta>0$, we denote by $\wt E_{m,\delta}$ the subset of the points $x\in E_m$ for which there exists $r_x>0$ as in Lemma \ref{lem:ver1}, so that $\omega^+\bigl(E_m\setminus \wt E_{m,\delta}\bigr)=0$. 

\vv

\begin{lemma}\label{lem:ver2}
Let $m\geq1$ and $\delta>0$. Let $x_0\in \wt E_{m,\delta}$ and $$0<r_0\leq \min(r_{x_0},1/m,c_1\dist(x^+,
\d\Omega^+)),$$ for some $c_1>0$ small enough (recall that $\omega^{\pm}=\omega^{x_{\pm}}_{\Omega^{\pm}}$ and $x^{\pm}\in \Omega^{\pm}$ are as in Definition \ref{def:adm}).
Suppose that the ball $B_{0}=B(x_{0},r_{0})$ is $a$-$P_{\gamma,\omega^+}$-doubling. Then for 
all $x \in G_m(x_{0},r_{0})$ it holds that
\begin{equation}\label{e:mainl2}
\RR_{*}(\chi_{2B_{0}}\omega^{+})(x)\leq C_1\,\Theta_{\mu}(B_{0}).
\end{equation}
\end{lemma}

\begin{proof}
To estimate  $|\RR_r(\chi_{2B_0}\omega^+)(x)|$ for $x \in G_m(x_{0},r_{0})$ we may assume that $r\leq r_0/4$ because 
 $|\RR_r(\chi_{2B_0}\omega^+)(x)|=0$ if $r\geq 4r_0$ and \rf{e:mainl2} is trivial
 in the case $r_0/4<r<4r_0$.

So we take  $x \in G_m(x_{0},r_{0})$ and $0<r\leq r_0/4$. First we turn our attention to $\RR_r\omega^+(x)$.
Since $\wt\beta_{\overline{\Omega^+},\infty}(x,r)
\leq 1/{100}$ (by the definition of $E_{m,\delta}$ and 
the fact that $r_0\leq 1/m$), 
 there is $x_{B}\in B:= B(x,r)$ with $B(x_{B},r/4)\subset B(x,r)\cap \Omega^-$. 
 Then, by \rf{eqclau2}, we have
\begin{equation}\label{e:r=k}
\RR\omega^{+}(x_{B})=K(x_{B}-x^{+}).
\end{equation}
By standard estimates, and because $B(x_{B},r/4)\subset B(x,r)\setminus\d \Omega^+$,
\begin{align}\label{e:lem2Rieszdiff}
|\RR\omega^{+}(x_{B})  &-\RR_{r}\omega^{+}(x)|\\
& = \av{\int_{B(x_{B},r/4)^{c}}\frac{x_{B}-y}{|x_{B}-y|^{n+1}}\,d\omega^{+}(y) -\int_{\overline B(x,r)^{c}}\frac{x-y}{|x-y|^{n+1}}\,d\omega^{+}(y)} \notag \\
& \lec \int_{\overline B(x,r)^{c}} \frac{|x-x_B|}{|x-y|^{n+1}}  d\omega^{+}(y)  \notag\\
&\quad + \int_{\overline B(x,r)\Delta B(x_{B},r/4)} \left(\frac1{|x_{B}-y|^n} +\frac{1}{|x-y|^n}\right)
\,d\omega^{+}(y) 
 \notag \\
& \lec P_{\omega^{+}}(B(x,r)).\nonumber
\end{align}

Using that $\overline B(x,r)\subset 2B_0$, we deduce that
 \begin{align*}
 |\RR_{r}(\chi_{2B_0}\omega^{+})(x)|
 & = \av{\int_{2B_{0}\backslash \overline B(x,r)}\frac{x-y}{|x-y|^{n+1}}d\omega^{+}(y)}\\
 & = \av{\int_{\overline B(x,r)^{c}}\frac{x-y}{|x-y|^{n+1}}d\omega^{+}(y)-\int_{2B_{0}^{c}}\frac{x-y}{|x-y|^{n+1}}d\omega^{+}(y)}\\
 &\leq |\RR_{r}\omega^{+}(x)-\RR_{r_0}\omega^{+}(x)| + C\,\Theta_{\omega^+}(2 B_0)\\
 & \lesssim |\RR\omega^{+}(x_{B})  -\RR\omega^{+}(x_{B_0})| +
 P_{\omega^{+}}(B_0)  +P_{\omega^{+}}(B(x,r)).
 \end{align*}
where $x_{B_0}$ is a point such that $B(x_{B_0},r/4)\subset B_0\cap \Omega^-$.
By \rf{e:r=k}, we have
$$|\RR\omega^{+}(x_{B})  -\RR\omega^{+}(x_{B_0})| = 
|K(x_{B}-x^{+}) - K(x_{B_0}-x^{+})|\lesssim \frac{r_0}{|x^+-x_0|^{n+1}}.$$
On the other hand, letting $N$ be the largest natural number such that $2^Nr\leq 2r_0$, by \rf{e:mainl1} we get
\begin{align}\label{eqkk232}
P_{\omega^{+}}(B(x,r)) 
& = \sum_{j\geq0}2^{-j} \Theta_{\omega^{+}}(B(x,2^{j}r))\\
& \!\!\stackrel{\eqn{mainl1}}{\lec} \sum_{0\leq j\leq N} 2^{-j}\Theta_{\omega^{+}}(B(x_0,r_0))+\sum_{j > N}2^{-j}\Theta_{\omega^{+}}(B(x_0,2^{j}r)) \nonumber\\
& \lesssim P_{\omega^{+}}(B_0).\nonumber
\end{align}
From the last estimates we infer that
$$|\RR_{r}(\chi_{2B_0}\omega^{+})(x)| \lesssim \frac{r_0}{|x^+-x_0|^{n+1}} + P_{\omega^{+}}(B_0).$$

Note now that
\begin{align}\label{eqnotnow1}
\frac{r_0}{|x^{+}-x_0|^{n+1}} &\stackrel{\eqn{bourgain}}{\lesssim} \frac{\omega^{+}(B(x_0,2\delta_0^{-1}|x^{+}-x_0|))}{|x^+-x_0|^{n}}\,\frac{r_0}{|x^+-x_0|}\\ 
& =\Theta_{\omega^+}(B(x_0,2\delta_0^{-1}|x^+-x_0|))\, \frac{r_0}{|x^+-x_0|} \lesssim P_{\omega^{+}}(B_0).\nonumber
\end{align}
 Therefore, recalling that $B_0$ is $a$-$P_{\omega^+}$-doubling,
  \begin{align*}
 |\RR_{r}(\chi_{2B_0}\omega^{+})(x)|&\lesssim P_{\omega^{+}}(B_0) \lesssim \Theta_{\omega^{+}}(B_0),
  \end{align*}
 which concludes \eqn{mainl2}.
\end{proof}
\vv

Let $m\geq1$, $\delta>0$, and $x_0\in \wt E_{m,\delta}$, and denote
$$G_m^{zd}(x_0,r_0) = \{x\in G_m(x_0,r_0): \lim_{r\to0}\Theta_{\omega^+}(B(x,r))=0\},$$
and
$$G_m^{pd}(x_0,r_0) = \{x\in G_m(x_0,r_0): \limsup_{r\to0}\Theta_{\omega^+}(B(x,r))>0\}.$$
The notation ``$zd$'' stands for ``zero density'', and ``$pd$'' stands for ``positive density''.

\vv

\begin{lemma}\label{lem:rectifac}
Let $m\geq1$ and $\delta>0$. Let $x_0\in \wt E_{m,\delta}$ and $$0<r_0\leq \min(r_{x_0},1/m,c_1\dist(x^+,
\d\Omega^+)),$$ for some $c_1>0$ small enough.
Suppose that the ball $B_{0}=B(x_{0},r_{0})$ is $a$-$P_{\gamma,\omega^+}$-doubling. Then there is 
an $n$-rectifiable
set $F(x_0,r_0)\subset G_m^{pd}(x_0,r_0)$ such that 
\[\omega^+(G_m^{pd}(x_0,r_0)\setminus F(x_0,r_0))=0\] 
and so that
$\omega^+|_{F(x_0,r_0)}$ and $\HH^n|_{F(x_0,r_0)}$ are mutually absolutely continuous.
\end{lemma}

\begin{proof}
From \rf{e:mainl1} we know that $\Theta_{\omega^+}(B(x,r))\lesssim \Theta_{\omega^+}(B(x_0,r_0))$ for all $x\in  G_m(x_0,r_0)$ and all $r\leq 2r_0$. Thus,
$$0<\limsup_{r\to0}\Theta_{\omega^+}(B(x,r))<\infty\quad \mbox{ for all $x\in  G_m^{pd}(x_0,r_0)$.}$$
Now the main the result from \cite{AHMMMTV15} asserts that $\omega^+|_{G_m^{pd}(x_0,r_0)}$ is 
$n$-rectifiable and proves the lemma.

An alternative argument consists in using the fact that $\RR_*\omega^+(x)<\infty$ for all such $x$ (by Lemma \ref{lem:ver2}) and then applying the 
Nazarov-Tolsa-Volberg theorem \cite{NToV-pubmat}.
\end{proof}

\vv
To deal with the set $G_m^{zd}(x_0,r_0)$ we intend to apply Theorem \ref{teo1}. The next lemma will be necessary to show that one
the key assumptions of that theorem is satisfied.
\vv

\begin{lemma}\label{lem:mainl3}
Let $m\geq1$ and $\delta>0$. Let $x_0\in \wt E_{m,\delta}$ and $$0<r_0\leq \min(r_{x_0},1/m,c_1\dist(x^+,
\d\Omega^+)),$$ for some $c_1>0$ small enough.
Suppose that the ball $B_{0}=B(x_{0},r_{0})$ is $a$-$P_{\gamma,\omega^+}$-doubling, then 
\begin{multline}\label{e:mainl3}
\int_{G_m^{zd}(x_0,r_0)}|\RR\omega^{+}(x)-m_{\mu,G_m^{zd}(x_0,r_0)}(\RR\omega^{+})|^2\,d\omega^{+}(x)\\
\lesssim \left(\frac{r_0}{|x^+-x_0|}\right)^{2-2\gamma}\, \Theta_{\omega^{+}}(B_0)^2\, \omega^{+}(B_0).
\end{multline}
\end{lemma}

\vv
Note that the integral in the left hand side of \rf{e:mainl3} is over $G_m^{zd}(x_0,r_0)$, the subset of zero density points
of $G_m(x_0,r_0)$. This is essential for the
validity of the estimate.

In \rf{e:mainl3}, $\RR\omega^{+}(x)$ should be understood in the principal value sense. The existence of this principal value
for $\omega^+$-a.e.\ $x\in G_m^{zd}(x_0,r_0)$ is guaranteed by the fact that, by Lemmas \ref{lem:ver1} and \ref{lem:ver2},
$$\sup_{0<r\leq 2 r_0} \frac{\omega^+(B(x,r))}{r^n} + \RR_*(\chi_{2B}\,\omega^+)(x)\lesssim 
\Theta_\mu(B(x_0,r_0)),$$
and then using Remark \ref{rempv}.

An alternative argument to prove the existence of the principal values is the following:
For $\omega^{+}$-a.e.\ $x\in G_m^{zd}(x_0,r_0)$, there is a sequence $r_{j}\to0$ so that each ball $B(x,r_{j})$ is $a$-$P_{\omega^+}$-doubling. By arguments analogous to the ones in the proof of Lemma \ref{lem:ver2}, one can show that 
for  $0<r<r'<r_{j}<r_0$, 
$$|\RR_{r}\omega^{+}(x) -\RR_{r'}\omega^{+}(x)|
\lec  P_{\omega^{+}}(B(x,r'))  + \frac{r'}{|x-x^{+}|^{n+1}}.$$
As in \rf{eqnotnow1} with $r'$ instead of $r_0$, it follows that
$$\frac{r'}{|x-x^{+}|^{n+1}}\lesssim P_{\omega^{+}}(B(x,r')).$$
Then, arguing as in \rf{eqkk232}, we have $P_{\omega^{+}}(B(x,r'))\lesssim P_{\omega^{+}}(B(x,r_j))$ and
using that $B(x,r_j)$ is $a$-$P_{\omega^+}$-doubling, we derive
$$|\RR_{r}\omega^{+}(x) -\RR_{r'}\omega^{+}(x)| \lesssim P_{\omega^{+}}(B(x,r_j)) \lec \Theta_{\omega^{+}}(B(x,r_{j}))
\to 0\quad\mbox{ as\; $j\to\infty$,}
$$
since $\lim_{r\to0}\Theta_{\omega^+}(B(x,r))=0$. By the Cauchy criterion, we infer that $\lim_{r\to0}\RR_{r}\omega^{+}(x)$ exists.
\vv

\begin{proof}[Proof of Lemma \ref{lem:mainl3}]
We claim that for $\omega^+$-a.e.\, $x\in G_m^{zd}(x_0,r_0)$, 
\begin{equation}
\RR\omega^+(x) = K(x-x^+).
\label{e:r=k2}
\end{equation}
Indeed, consider a sequence $r_{j}\rightarrow 0$ so that $B_{j}=B(x,r_{j})$ is $a$-$P_{\omega^+}$-doubling for every $j$. For $j$ large enough, we may find an $x_{B_{j}}\in B_{j}\backslash \Omega^{+}$ just as in Lemma \ref{lem:ver2}. Then 
\[ |K(x_{B_{j}}-x^{+}) -\RR_{r_{j}}\omega^{+}(x)| \stackrel{\eqn{r=k}}{=} |\RR\omega^{+}(x_{B_{j}})  -\RR_{r_{j}}\omega^{+}(x)|
 \stackrel{\eqn{lem2Rieszdiff}}{\lec} P_{\omega^{+}}(B_{j})\lec \Theta_{\omega^{+}}(B_{j}),\]
 and since $\lim_{r\rightarrow 0}\Theta_{\omega^{+}}(B(x,r))=0$, this implies \eqn{r=k2}. 
 
We deduce that
 \begin{equation}\label{eq**101}
|\RR\omega^+(x) - m_{\omega^+,G_m^{zd}(x_0,r_0)}(\RR\omega^+)| \leq \!\!\sup_{y\in G_m^{zd}(x_0,r_0)} \!|K(x-x^+) - K(y-x^+)|
\lesssim \frac{r_0}{|x_0-x^+|^{n+1}}.
\end{equation}
We will estimate the last term in the equation above by arguments analogous to the ones in \rf{eqnotnow1}, but now taking 
advantage of the fact that $B_0$ is $a$-$P_{\gamma,\omega^+}$-doubling with $\gamma<1$. So we write
\begin{align*}
\frac{r_0}{|x_0-x^{+}|^{n+1}} &\stackrel{\eqn{bourgain}}{\lesssim} \frac{\omega^{+}(B(x_0,2\delta_0^{-1}|x_0-x^{+}|))}{|x_0-x^{+}|^{n}}\,\frac{r_0}{|x_0-x^{+}|}\\ 
& =\Theta_{\omega^+}(B(x_0,2\delta_0^{-1}|x_0-x^{+}|))\, \left(\frac{r_0}{|x_0-x^{+}|}\right)^\gamma \left(\frac{r_0}{|x_0-x^{+}|}\right)^{1-\gamma}\\
& \lesssim P_{\gamma,\omega^{+}}(B_0)\left(\frac{r_0}{|x_0-x^{+}|}\right)^{1-\gamma}\\
& \lesssim \Theta_{\omega^{+}}(B_0)\left(\frac{r_0}{|x_0-x^{+}|}\right)^{1-\gamma}.
\end{align*}
The estimate \rf{e:mainl3} follows from \rf{eq**101} and the preceding inequality.
\end{proof}

\vvv
We can now finish the proof of \Theorem{main}. 
Recall that first we are assuming that $\omega^{+}\ll\omega^{-}\ll\omega^{+}$ on $E\subset \d\Omega^+$
and we wish to show that $E$ contains an $n$-rectifiable subset $F$ of full measure 
$\omega^+$ on $E$ 
on which $\omega^\pm$ are mutually absolutely continuous with respect to $\cH^{n}$. By standard arguments, it is enough
to show that for any subset $F_0\subset E$ with $\omega^+(F_0)>0$ there exists some
$n$-rectifiable subset $G_0\subset E_0$ with $\omega^+(G_0)>0$ on which $\omega^\pm$ are mutually absolutely continuous with respect to $\cH^{n}$.

Let $\delta>0$ be some small constant to be fixed below. 
By \rf{eq:EMM}, there exists some $m$ such that $\omega(F_0\cap E_m)>0$, which implies that
$\omega(F_0\cap \wt E_{m,\delta})>0$. Let $x_0$ be a point of $\omega^+$-density of 
$F_0\cap \wt E_{m,\delta}$ for which there exists a sequence of $a$-$P_{\omega^+}$-doubling
balls $B(x_0,r_j)$ with $r_j\to0$  and such that
$$\lim_{r\rightarrow 0}\beta_{\d\Omega^+,\infty}(x_0,r)=0$$
(by Lemma \ref{l:adoubling} and Theorem \ref{t:alpha0}
 such point $x_0$ exists).
 
 Let 
 $$0<r_j\leq \min(r_{x_0},1/m,c_1\dist(x^+,
\d\Omega^+))$$ 
be such that 
\begin{equation}\label{eqqq:1}
\beta_{\d\Omega^+,\infty}(x_0,r_j)\leq\delta
\end{equation}
and 
\begin{equation}\label{eqqq:2}
\omega^+(F_0\cap \wt E_{m,\delta}\cap B(x_0,r_j))\geq (1-\delta)\,\omega^+(B(x_0,r_j)).
\end{equation}
By Lemmas \ref{lem:ver1} and \ref{lem:ver2},
we have 
\begin{equation}\label{eqqq:3}
\omega^+ (B(x_0,r_j) \setminus G_m(x_0,r_j)) \leq \delta \,\omega^+(B(x_0,r_j))
\end{equation}
 and
\begin{equation}\label{eqqq:4}
\sup_{0<r<2r_j} \Theta_{\omega^+}(B(x,r)) +
\RR_{*}(\chi_{B(x_j,2r_j)}\omega^{+})(x)\leq C_2\,\Theta_{\omega^+}(B(x_j,r_j))
\end{equation}
for all $x\in G_m(x_0,r_j)$.
Clearly, from \rf{eqqq:2}
it follows that $\omega^+(F_0\cap B(x_0,r_j))\geq (1-\delta)\,\omega^+(B(x_0,r_j))$, or equivalently,
$\omega^+(B(x_0,r_j)\setminus F_0)\leq \delta\,\omega^+(B(x_0,r_j))$. Together with \rf{eqqq:3}, this yields
\begin{align}\label{eqqq:5}
\omega^+(G_m(x_0,r_j)\cap F_0) & = \omega^+(G_m(x_0,r_j))- 
\omega^+(G_m(x_0,r_j)\setminus F_0)\\
&
\geq (1-\delta) \,\omega^+(B(x_0,r_j)) - \omega^+(B(x_0,r_j)\setminus F_0)\nonumber\\
& \geq (1-2\delta)\, \omega^+(B(x_0,r_j)).\nonumber
\end{align}

By Lemma \ref{lem:rectifac}, in the case that $\omega^+(G_m^{pd}(x_0,r_j)\cap F_0)>0$, we are done because
the measure $\omega^+|_{G_m^{pd}(x_0,r_j)}$ is $n$-rectifiable, and then we can choose
the set $G_0$ to be equal to $G_m^{pd}(x_0,r_j)\cap F_0$ minus a set zero measure $\omega^+$.
In the case that
 $\omega^+(G_m^{pd}(x_0,r_j)\cap F_0)=0$, \rf{eqqq:5} tells us that
\begin{equation}\label{eqqq:6}
\omega^+(G_m^{zd}(x_0,r_j)\cap F_0) \geq (1-2\delta)\, \omega^+(B(x_0,r_j)).
\end{equation}
Further, by Lemma \ref{lem:mainl3}, given any arbitrary constant $\tau>0$, if $r_j=r_j(\tau)$
is small enough, we have
\begin{align}\label{eqqq:7}
\int_{G_m^{zd}(x_0,r_j)\cap F_0}|\RR\omega^{+}&(x)-m_{\mu,G_m^{zd}(x_0,r_j)\cap F_0}(\RR\omega^{+})|^2\,d\omega^{+}(x)\\
&\leq 
\int_{G_m^{zd}(x_0,r_j)}|\RR\omega^{+}(x)-m_{\mu,G_m^{zd}(x_0,r_j)}(\RR\omega^{+})|^2\,d\omega^{+}(x)\nonumber\\
&\leq C\,\left(\frac{r_j}{|x^+-x_0|}\right)^{2-2\gamma}\, \Theta_{\omega^{+}}(B(x_0,r_j))^2\, \omega^{+}(B(x_0,r_j))\nonumber\\
& \leq \tau\, \Theta_{\omega^{+}}(B(x_0,r_j))^2\, \omega^{+}(B(x_0,r_j)).\nonumber
\end{align}

For $r_j$ small enough, from \rf{eqqq:1}, \rf{eqqq:4}, \rf{eqqq:6} and \rf{eqqq:7} and the fact that $B(x_0,r_j)$ is 
$a$-$P_{\omega^+}$-doubling, one easily checks that the assumptions of Theorem \ref{teo1}
hold with $\mu = \omega^+$, $B=B(x_0,r_j)$, and $G_B=G_m^{zd}(x_0,r_j)\cap F_0$, with
$\delta$ replaced by $2\delta$. An immediate consequence of the theorem is that there exists
an $n$-rectifiable subset $G_0\subset G_m^{zd}(x_0,r_j)\cap F_0$ such that $\omega^+(G_0)>0$, as wished 
\footnote{In fact, it easily follows that for $\omega^+$-a.e.\ $x\in G_0$, $\lim_{r\to0}
\omega^+(B(x,r))r^{-n} >0$. So the case when $\omega^+(G_m^{pd}(x_0,r_j)\cap F_0)=0$ does not occur.}.

\vvv
To conclude the proof of \Theorem{main} it remains to show that, given a Borel set $E\subset \d\Omega^+$, 
$$\omega^+|_E\perp \omega^-|_E \qquad\Longleftrightarrow \qquad \cH^{n}(E\cap T)=0.$$
The fact that $\omega|_E^+\perp \omega^-|_E$ implies that $\cH^{n}(E\cap T)=0$ follows by standard
arguments. Indeed, the points in the set $T$ satisfy the cone property and thus
$\omega^+$ and $\omega^-$ are both mutually absolutely continuous with $\cH^{n}$ on a subset $T'\subset T$ with $\cH^{n}(T\backslash T')=0$. So
$$\omega^+|_{E\cap T'} \approx \HH^n|_{E\cap T'}\approx\omega^-|_{E\cap T'}$$
(here ``$\approx$'' denotes mutual absolute continuity), and so the statement
$$\omega^+|_E\perp \omega^-|_E$$
is false if $\HH^n(E\cap T)>0$.

Conversely, if $\omega^+|_E\perp \omega^-|_E$ does not hold, then there is some subset $F\subset E$
with $\omega^+(F)>0$ such that $\omega^+|_F$ and $\omega^-|_F$ are mutually absolutely continuous.
By the part of \Theorem{main} that we have already proved, there exists an $n$-rectifiable subset $G\subset F$ with 
$\omega^+(F\setminus G)=\omega^-(F\setminus G)=0$ such that $\omega^+$ and $\omega^-$ are both mutually
absolutely continuous with $\HH^n|_G$. Let $G_0\subset G$ be some subset with $0<\HH^n(G_0)<\infty$.
It is not hard to show that $\cH^{n}$-a.e. $x\in G_0$ is a tangent point for $\d\Omega^+$ since $\beta_{\d\Omega^{+},\infty}(B(x,r))\rightarrow 0$ and $\sup_{y\in B(x,r)\cap E}\dist(y,V)/r\rightarrow 0$ where $V$ is the approximate tangent $n$-plane for $G_0$ (see \cite[Chapter 15]{Mattila}. Hence, $\HH^n(E\cap T)\geq \HH^n(G_0\cap T)>0$.
This completes the proof of \Theorem{main}.

\vv

\section{Proof of Corollaries \ref{coro1} and \ref{coro2}}

\begin{proof}[Proof of Corollary \ref{coro1}]
Denote $\Omega^+=\Omega^1$ and
$\Omega^-=\bigl(\,\cnj{\Omega^{+}}\,\bigr)^{c}$. Let $\omega^+=\omega^1$ and let $\omega^-$ be the harmonic measure 
for $\Omega^-$ with pole $x^-\in \Omega^-$.

By the maximum principle we have $\omega^2 \ll\omega^-$ on $E$. So there exists some function 
$g\in L^1(\omega^-)$ such that
$\omega^2|_E = g\, \omega^-$. Hence if we set $G=\{x\in E:g(x)>0\}$, it turns out that
 $\omega^2(E\setminus G) =0$ and $\omega^2|_G$ 
and $\omega^-|_G$
are mutually absolutely continuous. 

Since $\omega^1|_E$ and $\omega^2|_E$ are mutually absolutely continuous, we infer that
$\omega^1(E\setminus G) =0$, too, and thus
$$\omega^+|_E = \omega^+|_G \approx \omega^-|_G,$$
where ``$\approx$' denotes mutual absolute continuity.
Hence, as $\Omega^+$ satisfies the assumptions of Theorem \ref{t:main}, it follows that
 there exists
some  $n$-rectifiable subset $F\subset G$ with $\omega^+(G\setminus F)=0$ on which $\omega^{+}|_F$ is are mutually absolutely continuous with respect to $\cH^{n}|_F$.
\end{proof}

\vvv

\begin{proof}[Proof of Corollary \ref{coro2}]
As in the previous proof, 
we denote $\Omega^+=\Omega^1$ and
$\Omega^-=\bigl(\,\cnj{\Omega^{+}}\,\bigr)^{c}$. Also, we let $\omega^\pm$ be the respective harmonic measures of $\Omega^\pm$.
We take $G$ as above, so that $\omega^2(E\setminus G) =0$ and $\omega^2|_G$ 
and $\omega^-|_G$
are mutually absolutely continuous. 

We deduce that $\omega^1|_E\perp\omega^2|_E$ if and only if  $\omega^+|_G\perp\omega^-|_G$.
By Theorem \ref{t:main} applied to $\Omega^+$ and $G$, this is equivalent to $\cH^{n}(G\cap T^1)=0$, where $T^1$ is the
set of tangents for $\partial\Omega^1 = \partial\Omega_+$.

Since $\omega^2(E\setminus G)=0$, using the cone property it is easy to check that 
$\cH^n((E\setminus G)\cap T^2)=0$, where $T^2$ is the
set of tangents for $\partial\Omega^2$.
Since $E$ is relative open in $\partial\Omega^1$ and $\partial\Omega^2$, we have
$T^1 \cap E = T^2\cap E =T\cap E,$
and thus
 $$\cH^n(E\cap T) =  \cH^{n}(G\cap T^1)+ \cH^n((E\setminus G)\cap T^2) =0.$$

Conversely, if $\cH^n(E\cap T)=0$, then
$$\cH^{n}(G\cap T^1) = \cH^{n}(G\cap T) \leq \cH^n(E\cap T)=0.$$
Thus $\omega^+|_G\perp\omega^-|_G$ by Theorem \ref{t:main}.
\end{proof}

\vv

\def\cprime{$'$}
\providecommand{\bysame}{\leavevmode\hbox to3em{\hrulefill}\thinspace}
\providecommand{\MR}{\relax\ifhmode\unskip\space\fi MR }
\providecommand{\MRhref}[2]{%
  \href{http://www.ams.org/mathscinet-getitem?mr=#1}{#2}
}
\providecommand{\href}[2]{#2}

\end{document}